\documentclass{amsart}

\usepackage{amsthm,amsmath,amssymb,epsfig,graphicx,mathtools}
\usepackage{hyperref}
\usepackage{dsfont}
\usepackage{color}
\usepackage{enumerate}
\usepackage{url}
\usepackage[table]{xcolor}
\usepackage[export]{adjustbox}

\usepackage[ulem=normalem,draft]{changes}
\usepackage{cancel}

\makeindex

\if {
\hoffset=-1in
\voffset=-1in
\topmargin=2,75cm
\headsep=1cm
\oddsidemargin=2cm
\evensidemargin=2cm
\textwidth=13cm
\textheight=18,5cm
\tolerance=2000
} \fi

\if{
\voffset=-1in
\topmargin=2,75cm
\headsep=1cm
\oddsidemargin=1cm
\evensidemargin=1cm
\textwidth=16cm
\textheight=20,5cm
\tolerance=2500
}\fi

\numberwithin{equation}{section}

\newtheorem{theorem}{Theorem}[section]
\newtheorem{lemma}[theorem]{Lemma}
\newtheorem{proposition}[theorem]{Proposition}
\newtheorem{corollary}[theorem]{Corollary}
\newtheorem{assumption}{Assumption}
\newtheorem{claim}{Claim}

\theoremstyle{definition}
\newtheorem{definition}{Definition}[section]

\theoremstyle{remark}
\newtheorem{remark}[theorem]{Remark}

\def\dd{{\rm d}}

\def\weight(#1,#2){c_{#1,#2}}

\newcommand{\norm}[1]{\left\lVert#1\right\rVert}

\if {

}{{\caption}}
} \fi

\newcommand\findem{\hfill{$\blacksquare$}\medskip}

\def\ph{\hat{p}}

\def\uh{\hat{u}}
\def\vh{\hat{v}}
\def\wh{\hat{w}}
\def\xh{\hat{x}}

\newcommand\Ih{\hat{I}}

\newcommand\Th{\hat{T}}

\newcommand\Wh{\hat{W}}

\def\hb{\bar{h}}

\def\pb{\bar{p}}

\def\ub{\bar{u}}
\def\vb{\bar{v}}
\def\wb{\bar{w}}
\def\xb{\bar{x}}
\def\yb{\bar{y}}

\def\ut{\tilde{u}}

\def\wt{\tilde{w}}

\def\St{\tilde{S}}

\def\calc{{\mathcal C}}

\def\cals{{\mathcal S}}

\def\cR{{\mathcal R}}

\def\C{\mathcal{C}}

\def\F{\mathcal{F}}
\def\G{\mathcal{G}}
\def\H{\mathcal{H}}
\def\I{\mathcal{I}}
\def\J{\mathcal{J}}

\def\L{\mathcal{L}}

\def\P{\mathcal{P}}

\def\U{\mathcal{U}}
\def\V{\mathcal{V}}
\def\W{\mathcal{W}}
\def\X{\mathcal{X}}

\def\eps{\varepsilon}

\def\nuh{\hat{\nu}}

\def\chib{\bar{\chi}}

\def\nub{\bar{\nu}}

\def\xib{{\bar\xi}}

\def\1B{{\bf  1}}

\def\min{\mathop{\rm min}}
\def\max{\mathop{\rm max}}

\def\half{\mbox{$\frac{1}{2}$}}

\def\1B{{\bf  1}}

\def\cR{\mathbb{R}}

\newcommand{\bbR}{{\mathbb R}}

\newcommand\be{\begin{equation}}
\newcommand\ee{\end{equation}}
\newcommand\ba{\begin{array}}
\newcommand\ea{\end{array}}
\newcommand{\bea}{\begin{eqnarray}}
\newcommand{\eea}{\end{eqnarray}}
\newcommand{\bean}{\begin{eqnarray*}}
\newcommand{\eean}{\end{eqnarray*}}

\def\disp{\displaystyle}

\title[The shooting algorithm for partially control-affine problems]{The shooting algorithm for partially control-affine problems with application to an SIRS epidemiological model}
\author{M.S. Aronna}
\address{Escola de Matem\'atica Aplicada FGV EMAp, Funda\c{c}\~ao Getulio Vargas, Rio de Janeiro, Brazil}
\email{soledad.aronna@fgv.br, \href{https://sites.google.com/view/aronna/home}{https://sites.google.com/view/aronna/home} }

\author{J.M. Machado}

\address{Universit\'e Paris-Saclay, Orsay, France}
\email{joao-miguel.machado@u-psud.fr}

\thanks{$^*$ The first author was supported by FAPERJ, CNPq and CAPES (Brazil) and by the Alexander von Humboldt Foundation (Germany). The second author was supported by CAPES (Brazil) and Fondation de Math\'ematiques Jacques Hadamard.}

\begin{document}

\begin{abstract}
	In this article we propose a shooting algorithm for partially-affine optimal control problems, this is, systems in which the controls appear both linearly and nonlinearly in the dynamics. Since the shooting system generally has more equations than unknowns, the algorithm relies on the Gauss-Newton method. As a consequence, the convergence is locally quadratic provided that the derivative of the shooting function is injective and Lipschitz continuous at the optimal solution. We provide a proof of the convergence for the proposed algorithm using recently developed second order sufficient conditions for weak optimality of partially-affine problems. We illustrate the applicability of the algorithm by solving an optimal treatment-vaccination epidemiological problem.
\end{abstract}
\maketitle

\section{Introduction}
\label{intro}
In this article we propose and study the convergence of a shooting algorithm for the numerical solution of optimal control problems governed by equations of the form
\begin{equation}
\label{eqintro}
	\dot{x}(t) = f_0(x(t), u(t)) + \sum_{i = 1}^{m}v_i(t)f_i(x(t), u(t)), \quad \text{ a.e. on $[0,T]$.}
\end{equation}
Note that when $m=0$ then a nonlinear control system arises and when the $f_i$'s do not depend on $u$, for all $i=0,\dots,m$,  then the resulting system is {\em control-affine} (we will call the latter {\em totally control-affine} to differentiate them from {\em partially control-affine systems}). In this article, however, we are particularly interested in the case where both $m$ and the dimension $l$ of $u$ are positive and then the two types of control appear.

This study is motivated by many models that emerge in practice in which the associated system is partially control-affine. Among them we can cite the followings: the Goddard's problem proposed in \cite{Goddard} and analyzed in Bonnans {\em et al.} \cite{martinon2009numerical}, other models for rocket motion studied in Lawden \cite{Law63}, Bell and Jacobson \cite{BelJac}, Goh \cite{GohThesis,goh2008optimal}, Oberle \cite{Obe77}, Azimov \cite{Azi05} and Hull \cite{Hul11}, an optimal hydrothermal electricity production problem investigated in Bortolossi {\em et al.} \cite{BPT02}, a problem of atmospheric flight considered by Oberle in \cite{oberle90}, and an optimal production process in Cho {\em et al.} \cite{ChoAbaPar93} and Maurer {\em et al.} \cite{MauKimVos05}. Regarding applications, in this article we analyse, in particular, an epidemiological model inspired from Ledzewicz and Sch\"attler \cite{ledzewicz2011optimal}, with treatment and vaccination as control policies (see Section \ref{SIRS_OC}).

For optimal control problems subject to the dynamics \eqref{eqintro}, with endpoint and control constraints, we propose a shooting algorithm and show that its local convergence is guaranteed if second order sufficient optimality conditions (established in Aronna \cite{Aronna2018}) hold. These second order conditions are written in terms of the second derivative of the associated Lagrangian function and are an extension of results proved in Dmitruk \cite{Dmi77} for control-affine systems. It is worth mentioning that these conditions rely on Goh transform \cite{goh1966second}. More details, references and timeline for second order conditions for partially control-affine and (totally) control-affine problems can be found in {\em e.g.} Aronna \cite{Aronna2018} and Aronna {\em et al.} \cite{ABDL12}, respectively.

Shooting-like methods applied to the numerical solution of partially control-affine problems can be found in Oberle \cite{OberleThesis, oberle90} and Oberle-Taubert \cite{ObeTau97}, where a generalization of the algorithm proposed by Maurer \cite{Mau76} for (totally) affine systems is given.
These works present  practical implementations of shooting algorithms, but they do not deal with the issue of convergence through optimality conditions.

\if{
The subject of second order optimality conditions for these partially affine problems have been studied by Goh in \cite{Goh66a,Goh67,GohThesis,Goh08}, Dmitruk in \cite{Dmi11}, Dmitruk and Shishov in \cite{DmiShi10}, Bernstein and Zeidan \cite{BerZei90},  and Maurer and Osmolovskii \cite{MauOsm09}.
The first works were by Goh, who introduced a change of variables in \cite{Goh66a} and used it to obtain optimality conditions in \cite{Goh66a,Goh66,GohThesis}, always assuming uniqueness of the multiplier. The necessary conditions we present imply those by Goh \cite{Goh66}, when there is only one multiplier.
Recently, Dmitruk and Shishov \cite{DmiShi10} analysed the  quadratic functional associated with the second variation of the Lagrangian function and provided a set of necessary conditions for the nonnegativity of this quadratic functional. Their results are consequence of a second order necessary condition that we present.
In  \cite{Dmi11}, Dmitruk proposed, without proof, necessary and sufficient conditions for a problem having a particular structure: the affine control variable applies to a term depending only on the state variable, i.e. the affine and nonlinear controls are  `uncoupled'. This hypothesis is not used in our work.
The conditions established here coincide with those suggested in Dmitruk \cite{Dmi11}, when the latter are applicable.
In \cite{BerZei90}, Bernstein and Zeidan derived a Riccati equation for the singular linear-quadratic regulator, which is a modification of the classical linear-quadratic regulator where only some components of the control enter quadratically in the cost function.
All of these articles use Goh's Transformation to derive their conditions; we use this transformation as well.
On the other hand, in \cite{MauOsm09}, Maurer and Osmolovskii gave a sufficient condition for a class of problems having one affine control 
subject to bounds and such that it is bang-bang at the optimal solution. This structure is not studied here since no control constraints are considered, i.e. our optimal control is suppose to be totally singular.

Regarding second order optimality conditions, we provide a pair of necessary and sufficient conditions for weak optimality of totally singular solutions. These conditions are `no gap' in the sense that  the sufficient condition is obtained from the necessary one by strengthening an inequality. We do not assume uniqueness of multiplier. When our second order conditions are applied to the particular cases having either $m=0$ or $l=0,$ they give already existing results as we point out along the article.

Among the applications of the shooting method to the numerical solution of partially affine problems we can mention the articles Oberle \cite{OberleThesis,oberle90} and Oberle-Taubert \cite{ObeTau97}. In these articles the authors use a generalization of the algorithm that Maurer \cite{Mau76} suggested for totally affine systems.
These works present interesting implementations of a shooting-like method to solve partially affine control problems having bang-singular or bang-bang solutions and, in some cases, running-state constraints are considered. No result on convergence is given in these articles. 

In this paper we propose a shooting algorithm which can be also used to solve problems with bounds on the controls. Our algorithm is an extension of the method for totally affine problems presented in Aronna et al. \cite{ABM11}.
We give a theoretical support to this method by showing that the second order sufficient condition above-mentioned ensures the local quadratic convergence of the algorithm.
}\fi

The article is organized as follows. In Section \ref{problem_statement} we give the statement of the problem,  the main definitions and assumptions, and state the first order optimality conditions. The differential-algebraic system (DAE) derived from the first order conditions is deduced and analized in Section \ref{differential_algebraic}, while the shooting algorithm used to solve this DAE is described in Section \ref{shooting_algorithm}. In Section \ref{second_order} we recall second order necessary conditions, and we state the main result of the article on convergence of the shooting algorithm in Section \ref{converge_shooting}. In Section \ref{simpler_control_constraints} we extend our analysis to problems with control constraints by means of an auxiliary unconstrained transformed problem. In Section \ref{Section_examples} we work out examples and solve them numerically.

\vspace{0.5cm}
\textbf{Notations.}
Throughout the text we shall omit the arguments of some functions  whenever the context is clear, {\em e.g.} the time dependence is frequently omitted. If $h$ is a function of time and some other variables, {\em i.e.} $h = h(t,x)$, the time derivative is frequently referred as $\dot{h}$. For partial derivatives with respect to other variables we write $D_xh,$ $h_x$ or $h_{x_i}$ if $x_i$ is a component of $x$. The same convention is adopted for higher-order derivatives. Given two differentiable vector fields $g,h:\mathbb{R}^n \to \mathbb{R}^n$, the {\em Lie bracket} between them is defined by
\begin{equation}
\label{Lie.Brackets}
[g,h] := D_x h(x)g(x) - D_x g(x)h(x).
\end{equation}
We use the same notation for functions depending on $u$ and $v$ as well; nevertheless, the derivatives are always taken w.r.t. $x$.

By $\mathbb{R}^k$ we denote the $k$-dimensional Euclidean real space, {\em i.e.} the space of $k$-dimensional column vectors with the usual euclidean norm; and by $\mathbb{R}^{k,*}$ its dual space consisting of $k$-dimensional row vectors. $\mathbb{B}$ denotes the open unitary ball of $\mathbb{R}^k$. By $L^p([0,T]; \mathbb{R}^k)$ we mean the Lebesgue space of functions with domain being the interval $[0,T]$ and taking values in $\mathbb{R}^k$; while $W^{q,s}([0,T]; \mathbb{R}^k)$ denotes the Sobolev spaces.

\section{Statement of the Problem and Assumptions}
\label{problem_statement}
We start with the control-unconstrained setting, the control-constrained case being left for Section \ref{simpler_control_constraints}. 
Considering the function spaces 
$\U := L^{\infty}([0,T];\bbR^l)$, 
$\V := L^{\infty}([0,T];\bbR^m)$ and
$\X := W^{1, \infty}([0,T]; \bbR^n)$, we define the optimal control problem in the Mayer form
\begin{align}
	 \label{objectivefunction}
	 \text{ minimize }&
	 \phi(x(0),x(T)) \\
	 \text{subject to} & \nonumber \\
	\label{state.dynamics}
	 &  \dot{x}(t) = f(x(t), u(t), v(t)), \ \ \text{a.e. on $[0,T]$},\\
	 \label{initial-final.constraints}
	 &  \eta_j(x(0),x(T)) = 0, \ \ \text{for $ j = 1, \cdots, d_{\eta}$}.
\end{align}
We let (OC) denote problem \eqref{objectivefunction}-\eqref{initial-final.constraints}, where $\phi \colon \bbR^{2n} \to \bbR$,   $\eta_j\colon \bbR^{2n} \to \bbR$, for $j = 1, \dots, d_{\eta}$, and the dynamics $f:\mathbb{R}^{n+l+m}\to \mathbb{R}^n$ is of the form
\begin{equation}
	 f(x, u, v) := f_0(x,u)+\sum_{i=1}^m v_if_i(x,u).
\end{equation}
We make the following assumption for the aforementioned functions.
\begin{assumption}
	\label{datafunctions.lipschitz}
	All data functions $f_0, f_1, \cdots, f_m$, $\eta$ and $\phi$ have Lipschitz continuous second order derivatives.
\end{assumption}

A {\em feasible trajectory} is a tuple $w := (x,u,v) \in \W := \X \times \U \times \V$ that verifies the state dynamics \eqref{state.dynamics} and the initial-final constraints \eqref{initial-final.constraints}. In order to state the {\em Pontryagin Maximum Principle (PMP)}, we consider the {\em costate space} $\X_* := W^{1, \infty}([0,T]; \bbR^{n,*})$. Given an element $\lambda = (\beta, p)\in \bbR^{d_{\eta},*} \times \X_*$, we define the {\em pre-Hamiltonian}
\begin{equation}
\label{Hamiltonian}
	H(x,u,v,p):= p\cdot\left(f_0(x,u)+\sum_{i=1}^m v_if_i(x,u)\right),
\end{equation}  
the {\em endpoint Lagrangian}
\begin{equation}
\label{inicial-finalLagrangian}
	\ell(x_0, x_T,\beta) := \phi(x_0,x_T) + \sum_{j = 1}^{d_{\eta}} \beta_j\eta_j(x_0,x_T),
\end{equation}
and the {\em Lagrangian} function
\begin{multline}
\label{Lagrangian}
		\L(w, \lambda) := \ell(x(0), x(T), \beta) 
		 \\\disp + \int_0^T p(t)\cdot\left(f_0(x(t),u(t))+\sum_{i=1}^m v_i(t)f_i(x(t),u(t)) - \dot{x}(t)\right) \dd t.
\end{multline}

Before stating the PMP, we specify the notion of optimality that will be used.
\begin{definition}[Weak minimum]
	\label{weakminimum}
	A feasible trajectory $\wh = (\xh, \uh, \vh) \in \W$ is said to be a {\em weak minimum} of problem (OC) if, for some $\varepsilon > 0$, it is optimal in the set of feasible trajectories $w = (x,u,v)$ that satisfy
	\begin{equation*}
		\norm{ x- \xh }_{\infty} + \norm{u - \uh }_{\infty} + \norm{v - \vh}_{\infty} < \varepsilon.
	\end{equation*}
\end{definition}

For the reminder of the article we shall fix a nominal feasible trajectory $\wh = (\xh, \uh, \vh)$ for which optimality conditions will be given. Whenever the arguments of a function are omitted, we mean that it is evaluated at such trajectory. For a proof of the Pontrygin's Principle we refer the reader to the original work from Pontryagin \cite{pontryagin1962mathematical} or the more recent monographs \cite{schattler2012geometric,vinter2010optimal}. 

\begin{theorem}[Pontryagin's Maximum Principle]
	\label{PMP}
	If $\wh = (\xh, \uh, \vh)$ is a weak minimum of (OC), then there exists a multiplier $\lambda = (\beta, p) \in  \bbR^{d_{\eta},*} \times \X_*$, satisfying the costate dynamics:
		\begin{equation}
		\label{costate_dynamics}
			\dot{p} = -D_xH(\wh, p), \quad a.e. \ on \ [0,T];	
		\end{equation}
	the transversality conditions:
		\begin{equation}
		\label{transversality.conditions}
			\begin{split}
				p(0) &= -D_{x_0}\ell(\xh(0),\xh(T),\beta),\\
				p(T) &= D_{x_T}\ell(\xh(0),\xh(T),\beta),
			\end{split}
		\end{equation}
	and the stationarity of the Hamiltonian
	\begin{equation}
	\label{stationarity.Hamiltonian}
		D_uH(\wh,p) = 0 \textnormal{ and } \ D_vH(\wh,p) = 0,  \quad a.e. \ on \ [0,T]. 
	\end{equation}	
\end{theorem}

An element $\lambda$ that satisfies the PMP for a trajectory $w \in \W$ is called a {\em multiplier} and the pair $(w,\lambda)$ is called an {\em extremal}. For a solution $w$ of (OC), we can, in general, expect a set of multipliers, instead of a single one. This is problematic for the shooting algorithm proposed later in this article, therefore we make the following assumption which guarantees uniqueness of multiplier \cite{pontryagin1962mathematical}. 

\begin{assumption}
	\label{qualification.constraints}
	The derivative of the mapping 
	\begin{equation}
		\begin{split}
			\hat\eta \colon \bbR^n \times \U \times \V &\to \bbR^{d_{\eta}}\\
						(x(0), u, v) &\mapsto \eta(x(0), x(T))
		\end{split}
	\end{equation}
	is onto. Here the vector $x$ is the solution to \eqref{state.dynamics} given the control $(u,v)$ and initial condition $x(0)$. 
\end{assumption}

Throughout the reminder of the article, Assumption \ref{qualification.constraints} shall be assumed without declaration and the trajectory $\hat w$ is supposed to satisfy the first order necessary conditions given by the PMP. Hence, in view of Assumption \ref{qualification.constraints}, $\hat w$ possesses a unique multiplier $\hat \lambda := (\hat \beta, \ph)$. 

\section{The Equivalent Differential-Algebraic System}
\label{differential_algebraic}

The Pontryagin Maximum Principle implies that the optimal state $\xh$ together with the multiplier $\ph$ are solutions of a DAE induced by equations \eqref{state.dynamics}, \eqref{initial-final.constraints}, \eqref{costate_dynamics}, \eqref{transversality.conditions}, and \eqref{stationarity.Hamiltonian}. The next step consists in showing that there exists a representation of the controls as a function of $x$ and $p$, in such way that one can eliminate them and transform the DAE into a {\em two-point boundary value problem} (TPBVP). This can be achieved by using the stationarity of the Hamiltonian along with a suitable strengthened version of the {\em Legendre-Clebsch conditions} and application of the {\em Implicit Function Theorem} (IFT). 
\if{
We use the stationarity of the Hamiltonian, \eqref{stationarity.Hamiltonian}, together with a suitable strengthened version of the Legendre-Clebsch second order necessary condition to write the system in the form
\begin{equation}
\tag{IFTsys}
\label{IFTsys}
	\begin{split}
	\Theta(\xi, \alpha) &= 0,\\
	D_{\alpha}\Theta(\xi, \alpha) &\succ 0.
	\end{split}
\end{equation}
In this general system \eqref{IFTsys}, $\alpha$ and $\xi$ play the role of the controls $(u,v)$ and the tuple $(x, p)$, respectively. After a system as \eqref{IFTsys} is assembled, the {\em Implicit Function Theorem} (IFT) is applied to find a representation of $\alpha$ in terms of $\xi$.}\fi

\subsection{Controls in feedback form}
The conventional Legendre-Clebsch condition assumes the form 
\begin{equation}
	\label{LegendreClebsch.mixed}
	\tag{LC}
	 \left(
	 \begin{array}{cc}
	 	H_{uu}(\wh,\ph) & H_{uv}(\wh,\ph)\\
	 	\\
	 	
	 	H_{vu}(\wh,\ph)
	 	 & H_{vv}(\wh,\ph)
	 \end{array}
	 \right)
	 \succeq 0.
\end{equation}
A proof of \eqref{LegendreClebsch.mixed} for the present setting can be found in Aronna \cite[Corollary 1]{Aronna2018}. Note that, since $H_{vv}(\wh, \ph) \equiv 0$ and $H_{vu} = H_{uv}^T$, condition \eqref{LegendreClebsch.mixed} holds if, and only if 
\begin{equation}
	\label{LegendreClebsch.mixed.equivalent}
	H_{uu}(\wh,\ph) \succeq 0 \text{ and } H_{uv}(\wh,\ph) = 0.
\end{equation}
Since the matrix in \eqref{LegendreClebsch.mixed} is singular we cannot apply the IFT to \eqref{stationarity.Hamiltonian} and obtain our desired representations of the controls. Instead, what one usually does is computing the time derivatives of the {\em switching function} $H_v$ that may depend explicitly on the controls (see {\em e.g.} Bryson and Ho \cite{BrysonHo75}).
In order to simplify the calculations involved in computing these derivatives, we consider a general formula for the time derivative of a product $p \cdot F,$ where $F:\mathbb{R}^n \times \mathbb{R}^m \to \mathbb{R}^n$ is a vector field. Employing the notation of Lie brackets given in the notation paragraph, we get
\begin{equation}
\label{genericvectorfield.D1}
\frac{\dd}{\dd t} \big( \ph \cdot F(\xh, \uh) \big) = \ph \cdot [f_0, F] + \sum_{i = 1}^m \vh_j \ph \cdot [f_i, F] + \ph \cdot D_u F \dot{\uh}.
\end{equation}
We obtain $\dot{H}_{v_i}$ by choosing $F = f_i$. Recalling that $H_{vu}=0$, we get
\begin{equation}
\label{switching.function.D1_preview}
\dot{H}_{v_j}(\wh,\ph) = \ph \cdot [f_0, f_j] + \sum_{i = 1}^m \vh_j \ph \cdot [f_i, f_j].
\end{equation}
As a consequence of the following Proposition \ref{Goh.condition}, equation \eqref{switching.function.D1_preview} does not depend explicitly of the linear control $v$. 
\begin{proposition}[Goh conditions]
	\label{Goh.condition}
	Assume that $\wh$ is a weak minimum. Then the following identities hold
	\begin{equation*}
	\ph \cdot [f_i, f_j] = 0, \quad \text{ for $i,j = 1, \dots, m$}.
	\end{equation*}
\end{proposition} 

Proposition \ref{Goh.condition} was proposed and proved by Goh \cite{Goh66}. A generalization that applies to the framework of the current paper was given by Aronna in \cite[Cor. 5.2]{Aronna2018} as a corollary of second order necessary conditions for optimality when the set of multipliers is a singleton (see also \cite{ABDL12} and \cite{frankowska2013pointwise}). In view of Proposition \ref{Goh.condition}, equation \eqref{switching.function.D1_preview} reduces to
\begin{equation}
\label{switching.function.D1}
\dot{H}_{v_i}(\wh,\ph) = \ph \cdot [f_0, f_i].
\end{equation}
By derivating the latter equation once more w.r.t. time, we obtain
\begin{equation}
\label{hamiltonian_2timedev}
	\ddot{H}_{v_i} = \ph \cdot \left[f_0,  [f_0, f_i]\right]+\sum_{j = 1}^m \vh_j \ph \cdot \left[f_j,  [f_0, f_i]\right] + \ph \cdot D_u[f_0, f_i]\dot{\uh}.
\end{equation} 
We aim at removing the dependence on $\dot{\uh}$ from \eqref{hamiltonian_2timedev}. This can be done by using the stationarity condition $H_u(\wh,\ph) = 0$. Assuming enough regularity, the total time derivative of this expression gives
\begin{equation}
\label{H_udot}
\dot{H}_u(\wh,\ph) = H_{ux} \dot{\xh} + H_{up} \dot{\ph} + H_{uu} \dot{\uh} = 0,
\end{equation}
where the term $H_{uv}\dot{v}$ vanishes in view of \eqref{LegendreClebsch.mixed.equivalent}. To make \eqref{H_udot} more rigorous, we make the following assumption on the controls.
\begin{assumption}[Regularity of the controls]
	\label{regularity.controls}
	The nonlinear control $\uh$ is continuously differentiable and the linear control $\vh$ is continuous. 
\end{assumption}
This assumption is not restrictive since it follows from the IFT, once we assume the strengthened generalized Legendre-Clesbch condition \eqref{LC-like.mixedcontrols} below. In fact, using equation \eqref{H_udot} and assuming the strengthened Legendre-Clebsch condition w.r.t. $u$, {\em i.e.} $H_{uu}\succ 0$, we can lose the dependence of $\dot{\uh}$, by using the IFT on \eqref{H_udot}, which yields
\begin{equation}
\label{representation.nonlin.derivative}
\dot{\uh} = \Gamma(\uh, \vh, \xh, \ph),
\end{equation}
for  a $\C^1$-function $\Gamma$.

Equation \eqref{representation.nonlin.derivative} shows that the dependence on $\dot{\uh}$ can be removed from \eqref{hamiltonian_2timedev}. We are now in position to formulate a system that can be used to achieve our desired representation.
Consider the mapping
{\small 
\begin{equation}
	\label{somemapping}
	(w, \lambda) \mapsto
	\left(\begin{array}{cc}
		H_u(w,p)\\
		\\
		-\ddot{H}_v(w,p)
	\end{array}\right),
\end{equation}
}
whose Jacobian w.r.t. $(u,v)$ at the extremal $(\wh, \hat\lambda)$ is
\begin{equation}
\label{somemapping_jacobian}
	\J :=
	\left(\begin{array}{cc}
	\disp
	H_{uu}(\wh,\ph) & H_{uv}(\wh,\ph) \\
	& \\
	\disp
	-\frac{\partial \ddot{H}_v}{\partial u}(\wh,\ph) &\disp  -\frac{\partial \ddot{H}_v}{\partial v}(\wh,\ph)
	\end{array}\right).
\end{equation}
To apply (IFT) to $H_u = 0, -\ddot{H}_v = 0$ and retrieve the controls, we assume the following {\em strengthened generalized Legendre-Clebsch condition}
\begin{equation}
	\label{LC-like.mixedcontrols}
	\tag{SLC}
	H_{uu}(\wh,\ph) \succ 0, \quad -\frac{\partial \ddot{H}_v}{\partial v}(\wh,\ph) \succ 0. 
\end{equation}

We get the following result.
 \begin{theorem}
	\label{OC-OSequivalence}
	Assume that \eqref{LC-like.mixedcontrols} holds. If $\wh$ is a {\em weak minimum} with associated multiplier $\hat\lambda$, then the optimal control $(\uh, \vh)$ admits the feedback form
	\begin{equation}
	\label{control.elimination}
		\uh = U(\xh, \ph) \quad \vh = V(\xh, \ph),
	\end{equation}
	where $U$ and $V$ are $\mathcal{C}^1$-functions.
	Furthermore, the extremal $(\wh, \hat\lambda)$ satisfies the {\em optimality system}
	\begin{equation}
	\label{optimality.system}
	\tag{OS}
	\left\{\begin{split}
	&\dot{x} = f(x,U(x,p), V(x,p)), \quad\text{ a.e. on $[0,T],$}\\
	&\dot{p} = -p\cdot D_xf(x,U(x,p), V(x,p)), \quad \text{ a.e. on $[0,T],$}\\
	&\eta_j(x(0), x(T)) = 0, \quad \text{ for } j = 1, \cdots, d_{\eta},\\
	&\left(p(0),p(T)\right) = \left(-D_{x_0}\ell,D_{x_T}\ell \right)(x(0), x(T),\beta),\\
	&H_v(x(T), U(x(T), p(T))) = 0,\quad \dot{H}_v(x(0), U(x(0), p(0))) = 0. 
	\end{split}\right.
	\end{equation}
\end{theorem}
\begin{proof} 
	From our previous discussion, since $H_{uu}\succ 0$, we can remove the dependence of $\dot{\uh}$ from $\ddot{H}_v$. Note that since $H_{uv} \equiv 0$,
	\begin{equation}
		\label{matrices}
		\mathcal{J}  = \left(\begin{array}{cc}
		H_{uu} & 0 \\
		\disp - \frac{\partial \ddot{H_v}}{\partial u}   & \disp - \frac{\partial \ddot{H_v}}{\partial v}
		\end{array}\right) = 
		\left(\begin{array}{cc}
		H_{uu} & 0 \\
		0      & \disp - \frac{\partial \ddot{H_v}}{\partial v}
		\end{array}\right)
		\left(\begin{array}{cc}
		I & 0 \\
		\disp \frac{\partial \ddot{H_v}}{\partial v} ^{-1}\frac{\partial \ddot{H_v}}{\partial u} & I
		\end{array}\right).
	\end{equation}
	Since the second matrix in \eqref{matrices} is invertible from \eqref{LC-like.mixedcontrols} and the third one is invertible by inspection, $\mathcal{J}$ is also invertible. Representation \eqref{control.elimination} follows from the IFT. 
	
	Moving on to \eqref{optimality.system}, note that it is derived from the PMP. However, the feedback forms in \eqref{control.elimination} are equivalent to $H_u = 0, \ddot{H}_v = 0$. To obtain the stationarity of the Hamiltonian w.r.t. $v$, we include the boundary conditions $H_v(T) = \dot{H}_v(0) = 0$. 
	We could have chosen other pair of boundary conditions, but this choice will simplify the presentation of the results that follow. 
\end{proof}

\subsection{Computing the Linear Controls}
To solve \eqref{optimality.system}, we need explicit analytical expressions for the controls in terms of $x$ and $p$. The nonlinear controls usually can be obtained from the stationarity $H_u = 0.$ We start by assuming that the representation $\uh = U(\xh, \ph)$ was already obtained. 

In the sequel we introduce the Poisson bracket notation. Given two functions $g,h$ that depend on $x,p$, the {\em Poisson bracket} is given by
\begin{equation}
	\label{poisson_bracket}
	\{g,h\} := D_xgD_ph - D_pgD_xh = \sum_{i = 1}^n\left(\frac{\partial g}{\partial x_i}\frac{\partial h}{\partial p_i} - \frac{\partial g}{\partial p_i}\frac{\partial h}{\partial x_i}\right).
\end{equation}

The following result is a direct consequence of this definition.
\begin{proposition}
	\label{poisson_bracket.proposition}
	Let $F = F(x,p,t)$ be a $\mathcal{C}^1$-function. Then
	\begin{equation}
		\frac{\dd}{\dd t}F(x,p,t) = \{F,H\} + \frac{\partial F}{\partial t},
	\end{equation}
	provided that $(x,p)$ follows the Hamiltonian dynamics $\dot{x}  = H_p, \,\, -\dot{p} = H_x.$
\end{proposition}

As a consequence of Proposition \ref{poisson_bracket.proposition}, if the optimal control $(\uh, \vh)$ admits a feedback representation $\uh = U(x,p)$, then 
\begin{equation}
	\label{u_dev}
	\dot{\uh} = \{U, H\} = \{U, p\cdot f_0\} + \sum_{j = 1}^{m}\vh_j \{U,p\cdot f_j \}.
\end{equation}
By substituting \eqref{u_dev} in equation \eqref{hamiltonian_2timedev}, we obtain, for $i,j = 1, \dots, m$, 
\begin{equation}
\label{kernel_equation.prev1}
\begin{split}
		\ddot{H}_{v_i} = \ \gamma_{i0} + \disp \sum_{j = 1}^{m}\vh_j\gamma_{ij} = 0,\quad\text{where }\gamma_{ij} := \ \disp \ph\cdot \left([f_j, [f_0, f_i]] + D_u[f_0, f_i]\{U,\ph\cdot f_j\}\right).
\end{split}
\end{equation}

\section{The Shooting Algorithm}
\label{shooting_algorithm}
A well-known method for solving TPBVPs is the {\em shooting algorithm}. Given an initial guess for the states and costates, the method iteratively adjusts these initial values in order to verify the boundary conditions.

Our goal is to numerically solve  \eqref{optimality.system} by applying a shooting algorithm. Note that the unknown multiplier $\beta$ is involved in the formulation of \eqref{optimality.system}, it is then included as a shooting variable as shown below.
\subsection{The shooting function}
We define the {\em shooting function} as follows.
\begin{definition}[Shooting function]
	\label{shooting.function}
	Let $\cals : \mathbb{R}^n \times \mathbb{R}^{n,*} \times \mathbb{R}^{d_{\eta}} =: D(\cals) \to \mathbb{R}^{d_{\eta}} \times \mathbb{R}^{2n + 2m}$ be the shooting function given by
	\begin{equation}
		\label{shooting.function.equation}
		(x_0, p_0, \beta) =: \nu \mapsto \cals (\nu) =
		\left(\begin{array}{c}
			\eta(x_0,x(T))\\
			p_0 + D_{x_0}\ell(x_0,x(T),\beta)\\
			p(T) - D_{x_T}\ell(x_0,x(T),\beta)\\
			H_v\left(x(T), U(x(T), p(T))\right)\\
			\dot{H}_v\left(x_0, U(x_0, p_0)\right)
		\end{array}\right),
	\end{equation}
where $(x,p)$ is the solution of the initial value problem
	\begin{equation}
		\label{dynamics.substituted}
		\begin{split}
			\dot{x} &= H_p(x,U(x,p), V(x,p),p), \quad x(0) = x_0,\\
			\dot{p} &= -H_x(x,U(x,p), V(x,p),p), \quad p(0) = p_0.\\	
		\end{split}
	\end{equation}
\end{definition}

Solving the differential-algebraic system \eqref{optimality.system} is equivalent to finding the roots of the shooting function $\cals$. Since the number of unknowns in $\cals(\hat\nu) = 0$ may be smaller than the number of equations, the Gauss-Newton method is a suitable approach. At each step the method updates the current approximation $\nu_k$ by 
\begin{equation}
	\label{nu.update}
	\nu_{k+1} \leftarrow \nu_k + \Delta_k, 
\end{equation}
where the increment $\Delta_k$ is computed by solving the linear approximation of the least squares problem
\begin{equation}
	\label{gauss-newton.approximation}
	\begin{array}{l}
		\disp \min_{\Delta \in D(\cals)} \left| \cals(\nu_k) + \cals'(\nu_k)\Delta \right|^2.
	\end{array}
\end{equation}
The solution of the linear regression \eqref{gauss-newton.approximation} is known to be
\begin{equation}
	\label{linear.regression.solution}
	\Delta_k = -\left( \cals '(\nu_k)^T\cals '(\nu_k)\right)^{-1}\cals '(\nu_k)^T\cals (\nu_k),
\end{equation}
provided the matrix $\cals '(\nu_k)^T\cals '(\nu_k)$ is non-singular. 
One can prove that the Gauss-Newton method \eqref{nu.update}-\eqref{linear.regression.solution} converges at least linearly as long as the derivative $\cals'(\hat\nu)$ exists and is injective. If in addition it is also Lipschitz continuous, the method converges locally quadratically (see {\em e.g.} Fletcher \cite{fletcher2013practical}, or alternatively Bonnans \cite{bonnans2006numerical}). \

\subsection{ Computation of the derivative of the shooting function}
In this paragraph we aim at obtaining a linearized differential system to be used afterwards to compute the derivative of the shooting function.

A general differential-algebraic control system can be written as
\begin{equation}
\label{differential-algebraic.general}
\left\{\begin{split}
\dot{\xi} &= \F(\xi, \alpha),\\
0 &= \G(\xi,\alpha),\\
0 &= \I(\xi(0), \xi(T)),
\end{split}\right.
\end{equation}
where $\F:\mathbb{R}^n\times\mathbb{R}^m \to \mathbb{R}^n,\, \G:\mathbb{R}^n\times\mathbb{R}^m \to \mathbb{R}^{d_{\G}}$ and $\I:\mathbb{R}^n\times\mathbb{R}^n \to \mathbb{R}^{d_{\I}}$ are $\C^1$-functions. The functions $\xi$ and $\alpha$ represent the tuple of states and costates and the control, respectively. Consider $\wt = (\tilde\xi, \tilde\alpha)$ a solution of \eqref{differential-algebraic.general}, then the {\em linearization of \eqref{differential-algebraic.general}} at $\wt$ is given by
\begin{equation}
\label{differential-algebraic.general.linearized}
\left\{\begin{split}
\dot{\bar{\xi}}&= D_{\xi}\F(\wt)\bar{\xi} + D_{\alpha}\F(\wt)\bar{\alpha},\\
0 &= D_{\xi}\G(\wt)\bar{\xi} + D_{\alpha}\G(\wt)\bar{\alpha},\\
0 &= D_{\xi_0}\I\left(\tilde{\xi}(0), \tilde{\xi}(T)\right)\bar{\xi}(0) +  D_{\xi_T}\I\left(\tilde{\xi}(0), \tilde{\xi}(T)\right)\bar{\xi}(T).
\end{split}\right.
\end{equation}

Let us apply this procedure to get the linearization of \eqref{optimality.system}.  We set $\xi:=(x,p)$, $\alpha:=(u,v)$  and $w := (\xi, \alpha)$. 
The linearized state and costate dynamics \eqref{state.dynamics}, \eqref{costate_dynamics} can be written as
\begin{align}
	\label{statedynamics.linearized}
	\dot{\xb} &= D_xf(w)\xb + D_uf(w)\ub + D_vf(w)\vb,\\ 
	\label{costatedynamics.linearized}
	\dot{\pb} &= -\left(\pb H_{xp} + \xb^T H_{xx} + \ub^T H_{ux} + \vb^TH_{vx}\right).
\end{align}
The endpoint conditions are also easily linearized, giving
{\small 
\begin{align}
\label{endpoint.linearized.constraints}
0 &= D\eta(\xh(0), \xh(T))(\xb(0), \xb(T)),\\
\label{transversality_0.linearized}
\pb(0) &= -\left(\xb^T(0) D^2_{x_0}\ell(\wh, \hat\beta) + \xb^T(T) D^2_{x_0x_T}\ell(\wh, \hat\beta) + \sum_{j = 1}^{d_{\eta}} \hat\beta_jD_{x_0}\eta_j \right),\\
\label{transversality_T.linearized}
\pb(T) &= \left(\xb^T(T) D^2_{x_T}\ell(\wh, \hat\beta) + \xb^T(T) D^2_{x_0x_T}\ell(\wh, \hat\beta) + \sum_{j = 1}^{d_{\eta}} \hat\beta_jD_{x_T}\eta_j \right). 
\end{align}}
The linearization of the other components of \eqref{shooting.function} gives
\begin{align}
\label{linHu}
\textnormal{Lin } H_u &= \pb D_uf + \xb^TH^T_{ux} + \ub^TH_{uu}\\
\label{linHv}
\textnormal{Lin } \ddot{H}_v &= \pb D_vf + \xb^TH^T_{vx}\\
\label{linHvT}
\left.\textnormal{Lin } H_v \right|_{t = T} &=  \left. \pb D_vf \right|_{t = T}+\left. \xb^TH^T_{vx} \right|_{t = T}\\
\label{lindotHv0}
\left.\textnormal{Lin } \dot{H_v}\right|_{t = 0} &=  \left. \frac{\dd}{\dd t}\right|_{t = 0} \left(\pb D_vf + \xb^TH^T_{vx}\right).
\end{align}
The linearized system \eqref{statedynamics.linearized}-\eqref{transversality_T.linearized}, \eqref{linHu}-\eqref{lindotHv0} is referred as (LS).
Finally, the evaluation of $\cals'$ in the direction $\bar{\nu} := (\xb_0, \pb_0, \bar{\beta})$ gives:
\begin{equation}
\label{shootingfunction.derivative.explicit}
\cals'(\hat{\nu}) \bar{\nu} =
\left(
\begin{array}{c}
D\eta(\xh(0), \xh(T))(\xb_0, \xb(T))\\
\pb_0 + \left[\xb^T_0 D^2_{x_0}\ell + \xb^T(T) D^2_{x_0x_T}\ell + \sum_{j = 1}^{d_{\eta}} \bar{\beta}_jD_{x_0}\eta_j \right]\\
\pb(T) - \left[\xb^T(T) D^2_{x_T}\ell + \xb^T(T) D^2_{x_0x_T}\ell + \sum_{j = 1}^{d_{\eta}} \bar{\beta}_jD_{x_T}\eta_j \right]\\ \disp
\left. \pb D_vf + \xb^TH^T_{vx} \right|_{t = T}\\
\left. \frac{\dd}{\dd t} \left(\pb D_vf + \xb^TH^T_{vx}\right) \right|_{t = 0}
\end{array}
\right).
\end{equation}

\section{Second Order Optimality Conditions}
\label{second_order}
In this section, we briefly review second order optimality conditions given in Aronna \cite{Aronna2018} which we will apply later to prove convergence of the shooting algorithm.

The optimality conditions will be presented in terms of the quadratic form
	\begin{multline}
	\label{secondvariation}
		 \Omega(\wb) := D^2_{(x_0,x_T)^2} \ell(\xb(0), \xb(T))^2 \\+ \disp  \int_{0}^{T} \left( \xb^TH_{xx}\xb
		    + \ub^TH_{uu}\ub + 2\xb^TH_{ux}\ub + 2\xb^TH_{vx}\vb + 2\vb^TH_{uv}\ub\right) \dd t,
	\end{multline}
or some transformed version of it. A well-known result around such quadratic form, obtained by means of a second order Taylor expansion, is that
\begin{equation}
\label{lagrangina.secondderivative}
	D^2\L(\wb)^2 = \Omega(\wb).
\end{equation}

We define the {\em critical cone} as
\begin{equation}
\label{criticalcone}
	\C := \left\{ \wb \in \W: \text{\eqref{statedynamics.linearized} and \eqref{endpoint.linearized.constraints}} \ \text{hold} \right\}.
\end{equation}
Since we are interested in stating second order sufficient conditions, we will require perturbations of the controls and states in $L^2$. Hence, we extend $\Omega$ to the function space $\W_2 := \X_2 \times \U_2 \times \V_2$, where $\X_2 := W^{1,2}([0,T]; \bbR^n),$ $\U_2 := L^2([0,T];\bbR^l)$ and $\V_2 := L^2([0,T];\bbR^m)$. The closure of $\C$ in $\W_2$ becomes
\begin{equation}
\label{criticalcone.W2}
	\C_2 := \left\{ \wb \in \W_2: \text{\eqref{statedynamics.linearized} and \eqref{endpoint.linearized.constraints} \ hold} \right\},
\end{equation}
and one has $\C = \C_2 \cap \W$. Hence $\C \subset \C_2$ and the inclusion is dense, as discussed in \cite{Dmi77}.

\subsection{Second Order Necessary Conditions of Optimality}
The following result holds.
\begin{theorem}[Second order necessary condition \cite{Aronna2018, MR1641590}]
	\label{SONC.theorem.C2}
	Suppose that $\wh$ is a weak minimum of problem (OC). Then
	\begin{equation}
	\label{SONC.general.C2}
	\Omega(\wb) \ge 0,\quad \text{ for all } \wb \in \C_2. 
	\end{equation}
\end{theorem} 
\if{
\begin{remark}
	\label{LegendreClesbch.corollary}
	As discussed in \cite{Aronna2018}, for the case when the set of multipliers is not a singleton, \ref{SONC.general.C2} can be extended by restricting the set of multipliers to a set with more information around the nominal trajectory. 
	
	The new second order necessary condition is given as
	\begin{equation}
		\label{maximum.SONC}
		\max_{\lambda \in \left({\rm co}\Lambda\right)^\#}\Omega[\lambda](\wb) \ge 0, \text{ for all } \wb \in \C_2,
	\end{equation}
	where ${\rm co} \Lambda$ denotes the convex hull of $\Lambda$ and the set $\left({\rm co}\Lambda\right)^\#$ can be characterized as
	\begin{equation}
		\left({\rm co}\Lambda\right)^\# = \{ \lambda \in {\rm co} \Lambda : H_{uu}[\lambda] \succeq 0 \ \text{and} \ H_{uv}[\lambda] = 0, \text{ a.e. on }[0,T] \}.
	\end{equation}
	
	When the set of multipliers is a singleton, \eqref{maximum.SONC} implies that the set $\left({\rm co}\Lambda\right)^\#$ is nonempty and coincides with $\Lambda$. As a consequence, the Legendre-Clebsch condition stated in \eqref{LegendreClebsch.mixed.equivalent} follows. Hence the term $H_{uv}$ can be omitted from the quadratic form $\Omega$.
\end{remark}
}\fi
To state second order sufficient conditions one can not rely on coercivity of $\Omega$ w.r.t. the controls since $H_{vv} \equiv 0$. In order to overcome this problem, the {\em Goh transform} is employed. The latter is a change of variables introduced by Goh  in \cite{goh1966second} and applied by him and other authors to derive second order conditions \cite{Goh66, Dmi77}. For the linearized system \eqref{statedynamics.linearized}, Goh transform is defined as
\begin{equation}
\label{transformation.Goh}
		\yb(t) := \disp \int_0^t \vb(\tau)\dd \tau,\quad \bar{\xi}(t):= \xb(t) - f_v(t)\yb(t),
		\quad	\text{for} \,\, t \in [0,T].
\end{equation}
One can easily check that the dynamics of the new variable $\bar{\xi}$ is given by
\begin{gather}
	\label{dynamics.xi}
	\dot{\bar{\xi}} = f_x\bar{\xi} + f_u\ub + B\yb, \quad \bar{\xi}(0) = \xb(0),\\
\label{matrixB}
	\text{where }B := f_xf_v - \frac{\dd}{\dd t}f_v,  
\end{gather}
and $B$ is well-defined since $u$ is differentiable as stated in Assumption \ref{regularity.controls}. 

We are interested in how the functional $\Omega$ and the critical cone are expressed in terms of the transformed variables $(\bar{\xi}, \ub, \yb)$. For this, consider a critical direction $ \wb \in \calc$. Note that $\xb(T) = \bar{\xi}(T) + f_v(T)\yb(T)$ and $\xb(0) = \bar{\xi}(0)$. Hence we introduce the new variable $\hb := \yb(T)$, which appears in the transformation of the quadratic functional through integration by parts and becomes a value that is independent of $\yb$ when passing to the limit in the $L^2$-topology. Equation \eqref{endpoint.linearized.constraints} can be rewritten as
\begin{equation}
\label{endpoint.linearized.constraints.Goh}
D \eta_j(\xh(0), \xh(T))\left(\bar{\xi}(0), \bar{\xi}(T) + f_v(T)\hb\right) = 0, \ \text{  for $j = 1, \cdots, d_{\eta}$},
\end{equation}
so that the critical cones $\calc_2$ and $\calc$ are respectively mapped into the sets
{\small 
\begin{gather}
\P_2 := \left\{ (\bar{\xi}, \ub, \yb, \hb) \in \W_2 \times \mathbb{R}^m : \yb(0) = 0, \yb(T) = \hb, \eqref{dynamics.xi} \text{ and }\eqref{endpoint.linearized.constraints.Goh} \text{ hold}\right\},\\
\P := \left(\P_2 \cap \W\right) \times \mathbb{R}^m.
\end{gather}}
The quadratic functional $\Omega$ {can also be written in terms of the new variables $(\bar{\xi}, \ub, \yb, \hb)$, and} takes the form
\begin{multline}	
	\label{second_variation.Gohtransform}
		\Omega_{\P}(\bar{\xi}, \ub, \vb, \yb, \hb):= g(\bar{\xi}(0), \bar{\xi}(T), \hb) + \disp\int_0^T\left(\bar{\xi}^TH_{xx}\bar{\xi} + 2\ub^TH_{ux}\bar{\xi} \right. \\
		\left. + 2\yb^TM\bar{\xi} + \ub^TH_{uu}\ub + 2\yb^TE\ub + \yb^TR\yb + 2\vb^TG\yb\right) \dd t,
\end{multline}
where
\begin{gather}
\label{matrixM}
	M:= f_v^TH_{xx}-\dot{H}_{vx} - H_{vx}f_x, \ \ E:= f_v^TH^T_{ux} - H_{vx}f_u,\\
	\label{matrixG}
	S:= \half \left(H_{vx}f_v + (H_{vx}f_v)^T\right), \ \ G:=  \half\left(H_{vx}f_v - (H_{vx}f_v)^T\right),\\
	\label{matrixR}
	R:= f_v^TH_{xx}f_v - (H_{vx}B + (H_{vx}B)^T) - \dot{S},\\
	\label{function.g}
	g(\bar{\xi}_0, \bar{\xi}_T , \hb):= D^2\ell(\bar\xi_0, \bar\xi_T + f_v(T)\hb)^2 + \hb^T(2H_{vx}(T)\bar\xi_T + S(T)\hb).
\end{gather}
For every critical variation $(\xb, \ub,\vb)$ and its respective transformed version $(\bar{\xi}, \ub, \yb, \yb(T))$, one can relate the quadratic functionals $\Omega$ and $\Omega_{\P}$ through integration by parts, as in \cite{Dmi77, Aronna2018}, obtaining
\begin{equation}
	\label{quadraticforms.relation}
	\Omega(\xb, \ub, \vb) = \Omega_{\P}(\bar{\xi}, \ub, \vb, \yb, \yb(T)).
\end{equation}
In view of latter identity, one can obtain optimality conditions in terms of $\Omega_{\P}$ and its extension to to $\W_2 \times \cR^m$ introduced below. 

An important issue is the presence of the term $2\vb^TG\yb$, which depends on the untransformed variation $\vb$. The expression of $G$ (see \eqref{second_variation.Gohtransform} and \eqref{matrixG}) gives
\begin{equation}
G_{ij} = - p\cdot[f_i,f_j].
\end{equation}
Hence, using Goh's conditions from Proposition \ref{Goh.condition}, the matrix $G$ vanishes and our quadratic form does not depend on $\vb$. The new quadratic form $\Omega_{\P_2}$, obtained from continuously extending $\Omega_{\P}$ to $\W_2 \times \cR^m$, assumes the form
{\small  
\begin{multline}
\label{second_variation.Gohtransform.P2}
\Omega_{\P_2}(\bar{\xi}, \ub, \yb, \bar h):= g(\bar{\xi}(0), \bar{\xi}(T), \bar h) \\+ \int_0^T\left(\bar{\xi}^TH_{xx}\bar{\xi} + 2\ub^TH_{ux}\bar{\xi} + 2\yb^TM\bar{\xi} + \ub^TH_{uu}\ub + 2\yb^TE\ub +  \yb^TR\yb \right) \dd t.
\end{multline}}
We are able now to state a version of necessary conditions which can be strengthened to sufficient conditions, once we assume coerciveness of $\Omega_{\P_2}$.
\begin{theorem}[\cite{Aronna2018}]
	If $\wh$ is a weak minimum of problem (OC), then
	\begin{equation}
	\label{SONC.Goh.setG.P2}
	\Omega_{\P_2}(\bar{\xi}, \ub, \yb, \hb) \ge 0, \quad \text{on} \ \P_2. 
	\end{equation}
\end{theorem}

\subsection{Second Order Sufficient Conditions of Optimality}
We introduce the following $\gamma$-order, which shall be used to state the sufficient conditions. For $(\xb(0), \ub, \yb, \hb) \in \mathbb{R}^n\times \U_2\times \V_2 \times \mathbb{R}^m,$ we define
\begin{equation}
	\gamma_{\P}(\xb(0), \ub, \yb, \hb) := |\xb(0)|^2 + \left|\bar h\right|^2 + \disp \int_0^T(|\ub(t)|^2 + |\yb(t)|^2) \dd t.	
\end{equation}
We can also express it as a function of the original variations by setting
\begin{equation*}
	\gamma(\xb(0), \ub, \vb) := \gamma_{\P}(\xb(0), \ub, \yb, \hb),
\end{equation*}
where $\yb$ is obtained from $\vb$ through Goh's transform \eqref{transformation.Goh} and $\hb:=\yb(T)$.
\begin{definition}[$\gamma$-growth]
	We say that a trajectory $\wh = (\xh, \uh, \vh)$ satisfies the {\em $\gamma$-growth condition in the weak sense} if there exist $\varepsilon, \rho >0$ such that
	\begin{equation}
		\label{g-growth}
		\phi(x(0), x(T)) \ge \phi(\xh(0), \xh(T)) + \rho\gamma(x(0) - \xh(0), u - \uh, v - \vh),
	\end{equation} 
for every feasible trajectory $w$ that verifies $\norm{w - \wh}_{\infty} < \varepsilon$.
\end{definition}

The following theorem was proved in \cite{Aronna2018} for a more general case allowing inequality endpoint constraints and possibly non-unique multiplier, and previously proposed by Dmitruk in \cite{Dmi77} in the totally control-affine setting.

\begin{theorem}[Sufficient condition for weak optimality \cite{Aronna2018}]
	\label{SOSC}
	Let $\hat w$ be a feasible trajectory satisfying the PMP with unique associated multiplier $\hat\lambda$. If for some $\rho > 0$ the quadratic functional $\Omega_{\P_2}$ satisfies
		\begin{equation}
		\label{Omega.P2.coersity.unique}
		\Omega_{\P_2}(\bar{\xi}, \ub, \yb, \hb) \ge \rho \gamma_{\P}(\xb(0), \ub, \yb, \hb) , \quad on \ \P_2,
		\end{equation}
		then  $\wh$ is a weak minimum satisfying the $\gamma$-growth in the weak sense.
		
		Conversely, if $\wh$ is a weak minimum satisfying $\gamma$-growth, then \eqref{Omega.P2.coersity.unique} is satisfied for some $\rho > 0$.
\end{theorem}
\begin{corollary}[\cite{Aronna2018}]
	\label{ColrecoverLCconditions}
	Let $\hat w$ be a feasible trajectory satisfying the PMP with unique associated multiplier $\hat\lambda$ and satisfying the coercivity condition \eqref{Omega.P2.coersity.unique}, then
	\begin{equation}
		\label{coercivity.matrix}
		\left(\begin{array}{cc}
			H_{uu} & E^T \\
			E      & R
		\end{array}\right) \succeq \rho I, \quad \textnormal{a.e. on $[0,T]$}. 
	\end{equation}
\end{corollary}

Goh stated in \cite{Goh66} that \eqref{coercivity.matrix} can be used to recover the strengthened Legendre-Clebsch condition \eqref{LC-like.mixedcontrols}. This result (see Proposition \ref{coercivity_implies_LC} below) is of great use since condition \eqref{LC-like.mixedcontrols} is necessary to obtain the controls in feedback form and assemble the optimality system (OS), as done in Theorem \ref{OC-OSequivalence}. To prove this implication we use the following Lemma \ref{goh_computations} that can be found in \cite{GohThesis, Goh66} and that was used in the literature by numerous authors. Nevertheless, since we believe that in Goh's work \cite{Goh66} there were some miscalculations, we included a revisited proof of Lemma \ref{goh_computations} in Appendix \ref{appendix_Goh.computations}.
\begin{lemma}
	\label{goh_computations}
	The following identities hold:
	\begin{equation}
	E = -\frac{\partial \dot{H}_v}{\partial u} \quad \text{ and } \quad R - EH_{uu}^{-1}E^T = -\frac{\partial \ddot{H}_v}{\partial v}.
	\end{equation}
\end{lemma}

\begin{proposition}
	\label{coercivity_implies_LC}
	Let $\wh$ be a feasible trajectory satisfying the coercivity condition \eqref{coercivity.matrix}. Then, the strengthened Legendre-Clebsch conditions, in the form of \eqref{LC-like.mixedcontrols}, hold.  
\end{proposition}
\begin{proof}
The argument is inspired by the discussion from Goh in \cite{goh2008optimal}. If the matrix in \eqref{coercivity.matrix} is positive definite then, for any $Q \in \mathbb{R}^{(l+m) \times (l+m)}$, we have
\begin{equation}
	a^TQ^T 
	\left(\begin{array}{cc}
	H_{uu} & E^T \\
	E      & R
	\end{array}\right)
	Qa > 0,
\end{equation}
provided that the vector $a$ is not in the kernel of $Q$. Therefore, in order for the product matrix to be positive definite, it suffices to choose $Q$ with full rank. 

Setting 
	$Q:={\small \left(\begin{array}{cc}
		I & -(H_{uu})^{-1}E^T \\
		0 & I
		\end{array}\right)^T}
	$, we check that
	{\small
		\begin{equation*}
		Q^T 
		\left(\begin{array}{cc}
		H_{uu} & E^T \\
		E      & R
		\end{array}\right)
		Q = 
		\left(\begin{array}{cc}
		H_{uu} & 0 \\
		0      & R - E(H_{uu})^{-1}E^T
		\end{array}\right) = \left(\begin{array}{cc}
		H_{uu} & 0 \\
		0      & - \disp \frac{\partial \ddot{H_v}}{\partial v}
		\end{array}\right),
		\end{equation*}}
	where the last equality comes from Lemma \ref{goh_computations}. Since the matrix $Q$ is non singular, \eqref{LC-like.mixedcontrols} follows.
\end{proof}

\section{Convergence of the Shooting Algorithm}
\label{converge_shooting}
Now we turn to the proof of convergence for the proposed shooting scheme. For this we formulate an auxiliary linear quadratic system as follows. 
\if{The logical steps of the proof will be as follows
{\color{red}
\begin{enumerate}[i]
	\item We define an auxiliary linear quadratic (LQ) problem that presents a single solution once we assume the coercivity condition \eqref{Omega.P2.coersity.unique} from Theorem \ref{SOSC};
	\item Propose a one-to-one transformation of Goh type that has the property of mapping the solutions of the linearization of (OS), the optimality system obtained from the original problem (OC), into the solutions of the optimality system that arises from the new auxiliary problem;
	\item The coercivity condition will imply that the linearization is around a weak minimum for the original problem; 
	\item Since the only solution from the linear quadratic system is null and the transformation is one-to-one the solution from the linearized original system is also null; 
	\item This will imply that the derivative of the shooting function is injective and the local convergence of the Gauss-Newton method will follow. 
\end{enumerate}}
}\fi
\subsection{The auxiliary linear quadratic problem}
Let (LQ) denote the optimal control problem defined by \eqref{LQ.cost}-\eqref{LQ.constraints} below
\begin{align}
	\label{LQ.cost}
	\text{minimize } &  \Omega_{\P_2}(\xib, \ub, \yb, \hb)\\
	\nonumber \text{subject to} & \\
	\label{LQ.dynamics}
	 & \dot{\bar{\xi}} = f_x\bar{\xi} + f_u\ub + B\yb,\\
	 \label{h.dyanmics}
	& \dot{\hb} = 0,\\
	\label{LQ.constraints}
	& 0 = D \eta_j(\xh(0), \xh(T))\left(\bar{\xi}(0), \bar{\xi}(T) + f_v(T)\hb\right),
\end{align}
where $\ub$ and $\yb$ denote the control variables, $\xib$ and $\hb$ are the states. Note that the feasible trajectories of (LQ) are the critical directions in $\P_2$. Once the coercivity condition \eqref{Omega.P2.coersity.unique} is assumed, the unique optimal solution of (LQ) is $(\bar\xi, \ub, \yb, h) = 0$.

In order to prove that the derivative of the shooting function $\cals$ is injective at a weak minimum, we exploit the correspondence between solutions of (LQ) and solutions of the linearized system (LS) (see Lemma \ref{LS-LQS_equivalence} below).

Let $\chib$ and $\chib_h$ denote the costates associated with $\xib$ and $\hb$, respectively. The qualification condition for the original problem given in Assumption \ref{qualification.constraints} easily translates into an analogous constraint qualification for problem (LQ). Consequently, the weak minimizer $(\xib, \ub, \yb, \hb) = 0$ of (LQ) also has a unique multiplier, which we shall refer as $\lambda^{LQ} := \left(\chib, \chib_h, \beta^{LQ} \right)$.

Define the pre-Hamiltonian for problem (LQ) and the endpoint Lagrangian as
{\small 
\begin{multline*}
		\H(\xib, \ub, \yb,\chib):= \chib (f_x\bar{\xi} + f_u\ub + B\yb) 
		 \\+ \half\bar{\xi}^TH_{xx}\bar{\xi}
		   + \ub^TH_{ux}\bar{\xi}
		 + \yb^TM\bar{\xi} + \half\ub^TH_{uu}\ub + \yb^TE\ub + \half\yb^TR\yb,\end{multline*}}
{\small 
\begin{equation*}
	\ell^{LQ}\left(\xib_0, \xib_T, \hb, \beta^{LQ}\right) := \half g(\xib_0, \xib_T, \hb)  
	+ \sum_{j = 1}^{d_\eta} \beta^{LQ}_jD \eta_j\left(\bar{\xi}_0, \bar{\xi}_T + f_v(T)\hb\right),
\end{equation*}}
respectively,
where $g$ was defined in \eqref{function.g}. The costate dynamics becomes
\begin{equation}
	\label{LQ.costatedynamics}
	-\dot{\chib} = \frac{\partial\H}{\partial \xib} = \chib f_x + \xib^TH_{xx} + \ub^TH_{ux} + \yb^TM, 
\end{equation}
with transversality conditions
{\small\begin{align}
	\label{LQ.costateinicial}
		\chib(0) & \disp 
				 = -\xib^T(0)D^2_{x_0^2}\ell + (\xib(T) + f_v(T)\hb)^TD^2_{x_0x_T}\ell + \sum_{j = 1}^{d_{\eta}}D_{x_0}\eta_j,\\
	\label{LQ.costatefinal}
		\chib(T) & \disp 
		=\xib^T(T)D^2_{x_T^2}\ell + \xib^T(0)D^2_{x_0x_T}\ell + \hb^TH_{vx}(T) + \sum_{j = 1}^{d_{\eta}}D_{x_T}\eta_j.
\end{align}}
The costate variable $\chib_h$ vanishes identically since $\dot{\chib}_h = 0$ and $\chib_h(0) = 0 $.
Finally, the stationarity of the Hamiltonian gives
\begin{align}
	\label{LQ.stationary_u}
	0 = \H_{\ub} &= \chib f_u + \xib^TH_{xu}^T + \ub^TH_{uu} + \yb^TE,\\
	\label{LQ.stationary_y}
	0 = \H_{\yb} &= \chib B + \xib^TM^T + \ub^TE^T + \yb^TR.
\end{align}
The set of equations \eqref{LQ.dynamics}-\eqref{LQ.constraints}, \eqref{LQ.costatedynamics}-\eqref{LQ.costatefinal} and \eqref{LQ.stationary_u}-\eqref{LQ.stationary_y} will be referred as the Linear Quadratic System (LQS). Notice that for this system, the matrix of the Legendre-Clebsch condition takes the form
\begin{equation}
	D_{(\ub, \yb)^2}^2\H =
	\left(
	\begin{array}{cc}
		H_{uu} & E^T\\
		E      & R
	\end{array}
	\right).
\end{equation}
Hence, if we assume coercivity for the original problem,  Corollary \ref{ColrecoverLCconditions} implies that $D_{(\ub, \yb)^2}^2\H$ is uniformly positive definite and then, solving the linear quadratic optimal control problem (LQ) is equivalent to solving its optimality condition (LQS).

\subsection{Linking the auxiliary problem with the optimality system}
Define the mapping
\begin{equation}
(\xb, \ub, \vb, \pb, \beta) \mapsto \left(\xib, \ub, \yb, \hb, \chib, \chib_h, \beta^{LQ}\right)
\end{equation}
through the equations
\begin{equation}
\label{mappingLS-LQS}
\begin{split}
\disp \yb(t) := \int_0^t\vb(s)\dd s, \qquad \xib := \xb - f_v\yb,\qquad \chib := \pb + \yb^{T}H_{vx},\\
\chib_h := 0, \qquad \hb := \yb(T), \qquad \beta^{LQ} := \beta.
\end{split}
\end{equation}
This Goh-type transformation is clearly one-to-one.
Recalling the linearization (LS) of the optimality system \eqref{optimality.system}, we show that this transformation maps solutions of (LS) into solutions of (LQS). Afterwards we shall use this property and the coercivity condition \eqref{Omega.P2.coersity.unique} to deduce the uniqueness of solution of (LS).

\begin{lemma}
	\label{LS-LQS_equivalence}
	If $\wh$ is a weak minimum of \textnormal{(OC)}, the injective mapping $(\xb, \ub, \vb, \pb, \beta) \mapsto (\xib, \ub, \yb, \hb, \chib, \chib_h, \beta^{LQ})$ defined in \eqref{mappingLS-LQS} converts solutions of \textnormal{(LS)} into solutions of \textnormal{(LQS)}.
\end{lemma}
The proof of this lemma is left for the Appendix \ref{appendix_Goh.computations}.

\if{
\begin{proof}
	We must check that given a solution $(\xb, \ub, \vb, \pb, \beta)$ of (LS), the corresponding transformed variables $(\xib, \ub, \yb, \hb, \chib, \chib_h, \beta^{LQ})$ solve (LQS).
	
	Starting with the state $\xib$, we recall the dynamics of the linearized variable $\xb$ given in \eqref{statedynamics.linearized} so that one has $\dot{\xib} = \dot{\xb} - \dot{f}_v\yb - f_v\dot{\yb} = f_x\xib + f_u\ub + B\yb,$ retrieving the dynamics in \eqref{LQ.dynamics}. The initial conditions are trivially satisfied since $\yb(0) = 0$. The dynamics for $\hb$ are satisfied by the definition. 
	For the costate dynamics we recall the dynamics of the linearized costates from \eqref{costatedynamics.linearized} and the definition of the matrix $M$ in \eqref{matrixM}. We get
	\begin{align*}
		-\dot{\chib} &= - \dot{\pb} - \dot{\yb}^TH_{vx} - \yb^T\dot{H}_{vx}\\
					 &= \underbrace{(\pb + \yb^TH_{vx})}_{=\chib}f_x + \underbrace{(\xb - f_v\yb)^T}_{=\xib^T}H_{xx} + \yb\underbrace{(f_v^TH_{xx} - \dot{H}_{vx} - H_{vx}f_x)}_{=M}\\
					 &= \chib f_x + \xib^TH_{xx} + \yb^TM.
	\end{align*}
	Hence the dynamics of $\chib$ matches \eqref{LQ.costatedynamics}. From equation \eqref{mappingLS-LQS} we obtain $\chib(0) = \pb(0)$ and deduce \eqref{LQ.costateinicial}. For the final conditions one substitutes the expressions for $\xb(T)$ and $\pb(T)$ into \eqref{linHvT} and conclude since $S = H_{vx}f_v = f^T_vH_{vx}^T,$	which is a consequence of the Goh conditions \eqref{Goh.condition}.This way we recover the transversality condition for $\chib(T)$. 
	
	Finally we must check the stationarity \eqref{LQ.stationary_u} and \eqref{LQ.stationary_y} of the Hamiltonian for (LQS). Starting from \eqref{linHu} and \eqref{mappingLS-LQS}, we obtain
	\begin{align*}
		0 &= (\bar\chi - \yb^TH_{vx})f_u + (\bar\xi + f_v\yb)^TH_{ux}^T +   \ub^TH_{uu}\\
		  &= \bar\chi f_u + \bar\xi^TH_{ux}^T + \ub^TH_{uu} + \yb^T(\underbrace{f_v^TH_{ux}^T - H_{vx}f_u}_{=E}),
 	\end{align*}
	which corresponds to the stationarity with respect to $\ub$. On the other hand, the same substitutions applied to \eqref{linHv} yield $0 = \chib f_v + \xib^TH_{vx}^T.$
	Differentiating with respect to time and using the definitions of $B$ in \eqref{matrixB} and $E$ in \eqref{matrixM}, we recover the stationarity \eqref{LQ.stationary_y} with respect to $\yb$. This shows that the tuple $(\xib, \ub, \yb, \hb, \chib, \chib_h, \beta^{LQ})$ is a solution of (LQS) and concludes the proof.
\end{proof}
}\fi 

\subsection{Convergence of the shooting algorithm}
We are in position to prove the convergence of the shooting algorithm given in \eqref{nu.update}-\eqref{linear.regression.solution}. We will use the following result on the behavior of the Gauss-Newton algorithm.
\begin{proposition}[\cite{bonnans2006numerical,fletcher2013practical}]
	\label{convergence_criteria_GN}
	If the matrix $\mathcal{S}'(\nuh)$ is injective, then the Gauss-Newton algorithm \eqref{nu.update}-\eqref{linear.regression.solution} is locally convergent. If in addition $\mathcal{S}'$  is Lipschitz continuous, then the algorithm converges locally quadratically.
\end{proposition}

The main result of this article is the theorem below that states a sufficient condition for the local quadratic convergence of the shooting algorithm.

\begin{theorem}[Convergence of the shooting algorithm]
	\label{convergence.unconstrained-problem}
	Let $\hat w$ be a feasible trajectory satisfying the \textnormal{PMP} that verifies the coercivity condition \eqref{Omega.P2.coersity.unique}. Then the shooting algorithm is locally quadratically convergent.
\end{theorem}
\begin{proof}
	From Theorem \ref{SOSC}, the trajectory $\wh$ is a weak minimum for problem (OC). From Corollary \ref{ColrecoverLCconditions} and Proposition \ref{coercivity_implies_LC}, \eqref{LC-like.mixedcontrols} holds. Consequently, Theorem \ref{OC-OSequivalence} implies that (OS) is well-posed  so that we can properly formulate the shooting algorithm. Therefore, consider some solution $(\xb, \ub, \vb, \pb, \beta)$ of (LS), and the associated transformed process $(\xib, \ub, \yb, \hb, \chib, \chib_h, \beta^{LQ})$ given by \eqref{mappingLS-LQS}. The latter is a solution of (LQS) in view of Lemma \ref{LS-LQS_equivalence}.
	However, once we assume condition \eqref{Omega.P2.coersity.unique}, the unique solution to (LQS) is the null trajectory and, since the transformation \eqref{mappingLS-LQS} is one-to-one, the solution to (LS) is also null.
	But from equation \eqref{shootingfunction.derivative.explicit}, the vectors $\nub$ in the kernel of $\cals'(\nuh)$ are precisely the solutions of (LS). We conclude that $\cals'(\nuh)$ is injective. In addition $\cals'$ is Lipschitz continuous due to Assumption \ref{datafunctions.lipschitz}. The  claim follows from Proposition \ref{convergence_criteria_GN}.
\end{proof}

\section{Control-constrained Problems}
\label{simpler_control_constraints}
In this section we extend the proposed algorithm to problems where the controls are subject to bounds. We denote by \textnormal{(CP)} the problem obtained by adding the following control constraints to \textnormal{(OC)}:
\begin{equation}
	\label{constraint.control}
	\begin{split}
		u(t) \in U, & \text{ a.e. on $[0,T],$ }\\
		0 \le v_i(t) \le 1,& \text{ a.e. on $[0,T],$ for $i = 1, \cdots, m$},
	\end{split}
\end{equation}
where $U$ is an open subset of $\mathbb{R}^l$. The choice of the bounds $0$ and $1$ was made for clarity of the exposition since it simplifies the notation, however all the results here presented hold for controls satisfying $a_i \le v_i \le b_i$ for any pair of bounds such that $a_i < b_i$.  Consider the following definition. 
\begin{definition}
	The component $\vh_i$ is said to have a {\em singular arc} in an interval $I$, whenever $0 < \vh_i(t) < 1$ a.e. on $I$. 
	If a component $\vh_i$ assumes the value $1$ (resp. $0$) a.e. on an interval $I$, it is said to have an {\em upper bang arc} (resp. {\em lower bang arc}) on this interval.  
	If $\vh_i$ has either an upper or a lower bang arc on $I$ then we can say, shortly, that it has a {\em bang arc} on $I$.\\
\end{definition}
\begin{assumption}
	We assume the following hypotheses on the optimal $(\uh, \vh)$.
\begin{itemize}
	\item[(i)] Each linear control $\vh_i$, with  $i=1,\dots,m$, presents a {\em bang-singular structure}, i.e. $\vh_i$ is a finite concatenation of {\em bang and singular arcs}.
	
	\item[(ii)] The bang-singular structure of $\vh$ induces a partition of the time interval $[0,T]$, that we write as
	\begin{equation*}
	\{0 := \Th_0 < \Th_1 < \Th_2 < \cdots < \Th_{N-1} < \Th_N := T\}.
	\end{equation*}
At each interval $\Ih_k := [\Th_k, \Th_{k+1}]$, every component $\vh_i$ is either bang or singular, and at $\Th_k$ some control $\vh_i$ switches its arc type, and presents a discontinuity of first kind. Hence, defining the sets
	\begin{equation*}
		\begin{split}
			S_k &:= \{ 1\le i \le m: \vh_i \text{ is singular on }\Ih_k \},\\
			A_k &:= \{ 1\le i \le m: \vh_i = 0 \text{ a.e. on }\Ih_k \},\\
			B_k &:= \{ 1\le i \le m: \vh_i = 1 \text{ a.e. on }\Ih_k \},
		\end{split}
	\end{equation*}
	 there must exist some $\rho' > 0$ such that 
	\begin{equation}
	\label{control_discontinuities}
	\rho' < \vh_i(t) < 1-\rho', \quad \text{ for all }i \in S_k,
	\text{ a.e. on $t \in \Ih_k$.}
	\end{equation}
	In addition, we assume that the nonlinear control satisfies
	\begin{equation}
		\label{control_set.rho-Ball}
		\uh([0,T]) + \rho'\mathbb{B} \subset U.
	\end{equation}
	
	\item[(iii)] \label{LC-constrained}	
	For each $k = 1, \dots, N$,  let $v_{S_k}$ denote the vector with components $v_i$ with $i \in S_k$. To obtain a feedback representation in a similar manner as done in Section \ref{differential_algebraic}, we assume that 
	\begin{equation}
	\label{controls_continuity}
	\begin{array}{c}
		\text{$\uh$ is continuously differentiable in $[0,T]$,}\\
		\text{$\vh_{S_k}$ is continuous in  $\Ih_k$, for $k=1,\dots,N$,}
	\end{array}
	\end{equation}
	and that, on each interval $\Ih_k$, the following form of the generalized strengthened Legendre-Clebsch conditions holds
	\begin{equation*}
		H_{uu}(\wh, \ph) \succ 0, \quad -\frac{\partial \ddot{H}_{v_{S_k}}}{\partial v_{S_k}}(\wh,\ph) \succ 0.
	\end{equation*}
\end{itemize}
\end{assumption}
Assumptions (i) and (ii) can be justified, in the context of control-affine problems, by the theory of {\em junction conditions}. More precisely, in \cite{mcdanell1971necessary}, McDanell and Powers established that when a singular arc has odd order $q$, the singular controls are obtained from the $2q$-{\em th} time derivative of the switching function, and the junction between singular and bang arcs is either $\mathcal{C}^1$ or discontinuous. Many examples found in the literature fall in the latter category (see {\em e.g.} \cite{Mau76,LEDSCH2014}), as well as the two examples treated in the Section \ref{Section_examples} of this work. For further details and examples concerning junction conditions, we refer to \cite{bell1993optimality}. The regularity of the controls assumed in (iii) comes as a consequence of the feedback representation derived in Section \ref{differential_algebraic}.

As a consequence of the minimization of the Hamiltonian given by the PMP, if a component $v_i$ is singular in some interval $I$, then $H_{v_i}(t) = 0$ a.e. on $I$ additionally to $H_u(t)=0.$ Hence, as done in Section \ref{differential_algebraic}, we can use the system
\begin{equation}
\label{statcons}
\left(\begin{array}{c}
H_u\\
-\ddot{H}_{v_{S_k}}
\end{array}\right) = 0,\quad \text{a.e. on } \Ih_k.
\end{equation}
along with item (iii) from Assumption \ref{LC-constrained} to write the controls $\uh$ and $\vh_{S_k}$ in feedback form, which we represent as
\begin{equation}
	\label{control_feedback_mappings}
	\uh = U(\xh,\ph), \quad \vh_{S_k} = V_{S_k}(\xh, \ph), \text{ for $k = 1, \dots, N.$}
\end{equation}

\subsection{The transformed problem}
Given a feasible control $(\uh,\vh)$, we call {\em control structure} the configuration of bang and singular arcs of $\vh$. In (CP), there may be feasible trajectories with a bang-singular structure different from the one of $(\uh,\vh).$ However, if $(\uh,\vh)$ is a local solution for (CP), it will also be a local solution for a problem with a fixed control structure. We assume {\em a priori} knowledge of the optimal control structure to formulate a new unconstrained problem whose feasible controls correspond to controls of the original problem that have such fixed structure. This is achieved by a reparametrization from $[0,T]$ to the interval $[0,1]$ as described next.

In this new unconstrained problem, for each switching time we associate a state variable $T_k$ having null dynamics, keeping the convention that $T_0 = 0$ and $T_N = T$. Such variables are initialized in the algorithm as a rough estimate of the optimal switching times, that will be iteratively tunned by the shooting scheme. For each interval $I_k := [T_k, T_{k+1}]$, we also associate a state variable $x^k$, that is the reparametrization of the state restricted to $I_k$ to the interval $[0,1]$.

The control variables of the new problem are defined as follows. For each interval $I_k$ of the partition we define a control variable $u^k\colon I_k\to \cR^l$ that appears nonlinearly  and an affine control $v^k\colon I_k \to \cR^{|S_k|}$. This way, each $v^k$ has as many entries as the number of singular components of $\vh$ in $\Ih_k$. The bang components of $v$ appear as constants and not as control variables, {\em i.e.} are fixed to either $0$ or to $1$. 

The trajectories of the transformed problem have the form
{\small
\begin{equation}
\label{Wh}
W := \left(\left(x^k\right)_{k=1}^N, \left(u^k\right)_{k=1}^N, \left(v^k\right)_{k=1}^N, \left(T_k\right)_{k=0}^{N}\right),
\end{equation}}
and the transformed problem, denoted as (TP), is the following:
{ \small 
\begin{align*}
\min 
&\,\phi(x^1(0), x^N(1)) \\
\text{s.t. } &  \dot{x}^k = (T_k - T_{k-1})\left(\sum_{i\in B_k\cup \{0\}}f_i(x^k,u^k) + \sum_{i\in S_k} v^k_if_i(x^k,u^k)\right), \ \ k = 1,\cdots,N ,\\
&\dot{T}_k = 0, \ \ k = 1,\cdots,N-1 ,\\
&\eta(x^1(0), x^N(T)) = 0, \ \ \\
&x^k(1) = x^{k+1}(0), \ \ k = 1,\cdots,N-1.
\end{align*}
}
Note that given some admissible trajectory $(x,u,v)$ of (CP), and its associated switching times $\left(T_k\right)$, we can obtain a feasible trajectory for (TP) via the following transformation
\begin{equation}\label{transform.CP-TP}
\begin{split}
x^k(s) &:= x\left(T_{k-1} + (T_k - T_{k-1})s\right),\\
u^k(s) &:= u\left(T_{k-1} + (T_k - T_{k-1})s\right),\\
v^k(s) &:= v\left(T_{k-1} + (T_k - T_{k-1})s\right).
\end{split}
\qquad \text{ for $s \in [0,1]$}.
\end{equation}
In fact, we discuss below that we can derive the weak optimality of a solution for (TP) from the optimality, in an appropriate sense, of a solution for (CP). To do this, consider the definition of {\em Pontryagin minimum} \cite{MR1641590} given next.
\begin{definition}
	\label{pontryagin_minimum}
	A feasible trajectory $\wh \in \W$ is a {\em Pontryagin minimum} of \textnormal{(CP)} if, for any positive $N$, there exists some $\varepsilon_N > 0$ such that $\wh$ is a minimum in the set of feasible trajectories $w = (x,u,v) \in \W$ satisfying
	\begin{equation*}
	\norm{x - \xh}_{\infty} < \varepsilon_N, \ \norm{(u,v) - (\uh,\vh)}_{1} < \varepsilon_N, \ \norm{(u,v) - (\uh,\vh)}_{\infty} < N.
	\end{equation*}
\end{definition}
\begin{lemma}
	\label{CP-TP-relation}
	If $\wh$ is a Pontryagin minimum of \textnormal{(CP)}, then $\Wh$ obtained from $\wh$ using transformation \eqref{transform.CP-TP} is a weak minimum of \textnormal{(TP)}.
\end{lemma}
The proof of this lemma follows as a direct extension of a similar result for the totally control-affine case given in \cite{aronna2013shooting}. For the sake of completeness, we included the proof in Appendix \ref{appendix_Goh.computations}.
\if{
\begin{proof}
Since $\wh$ is a Pontryagin minimum of (CP), from Definition \ref{pontryagin_minimum}, there exists $\varepsilon > 0$ such that
\begin{equation}
\label{ineqw}
	\norm{x - \xh}_{\infty} < \varepsilon, \ \norm{(u,v) - (\uh,\vh)}_{1} < \varepsilon, \ \norm{(u,v) - (\uh,\vh)}_{\infty} < 1.
\end{equation}
Let $\Wh$ be the transformation of $\wh$ through \eqref{transform.CP-TP}. We now prove that $\Wh$ is weakly optimal for (TP). Hence we search appropriate $\bar\delta, \bar\varepsilon$ for which all feasible trajectories $W = \big((x^k), (u^k), (v^k), (T_k)\big)$ of (TP) that satisfy
\begin{equation}
\label{weak_optimality_W}
\left| T_k - \Th_k\right|<\bar\delta, \ \norm{(u^k,v^k) - (\uh^k,\vh^k)}_{\infty} < \bar\varepsilon, \ \text{ for all } k = 1, \cdots, N 
\end{equation}
will be mapped into neighborhood of $\wh$ where it is optimal. Such mapping of $W \mapsto w$ is done as follows
\begin{gather}
	\label{transform.TP-CP.xu}
	x(t) := x^k\left(\frac{t - T_{k-1}}{T_k - T_{k-1}}\right), \quad u(t) := u^k\left(\frac{t - T_{k-1}}{T_k - T_{k-1}}\right),\quad \text{ for $t \in I_k$},\\
	\label{transform.TP-CP.v}
	v_i(t) :=
	\left\{
	\begin{array}{cc}
		0,& \text{if $t \in I_k$ and $i\in A_k,$}\\ 
		v_i^k\left(\frac{t - T_{k-1}}{T_k - T_{k-1}}\right),& \text{if $t \in I_k$ and $i\in S_k,$}\\
		1,& \text{if $t \in I_k$ and $i\in B_k.$}\\
	\end{array}
	\right.
\end{gather}

The dynamics \eqref{state.dynamics} are clearly satisfied by $(x,u,v)$ obtained from \eqref{transform.TP-CP.xu}-\eqref{transform.TP-CP.v}. The end-point constraints in \eqref{initial-final.constraints} are also easy to verify since $x(0) = x^1(0)$ and $x(T) = x^N(1)$ along with the feasibility of $W$.

The last step to check feasibility of $w$ are the control constraints. For the nonlinear controls, note that since $\norm{u^k - \uh^k}_{\infty} < \bar\varepsilon$, we have that $\norm{u - \uh}_{\infty} < \bar\varepsilon$. Recalling $\rho$ given in \eqref{control_discontinuities}-\eqref{control_set.rho-Ball}, if we choose $\bar\varepsilon < \rho$, then $u\left([0,T]\right) \subset U$.
To discuss the feasibility of the linear controls, from equation \eqref{control_discontinuities}, we can choose $\bar\varepsilon$ so that, whenever $t \in I_k$ and $i \in S_k,$
\begin{equation}
0 < \rho - \bar\varepsilon \le v_i(t) \le 1 -\rho + \bar\varepsilon < 1.
\end{equation}
On the other hand, for $i \in A_k \cup B_k$, we know that $v_i(t) \in \{0,1\}$ in view of \eqref{transform.TP-CP.v}, so that the control constraints are still satisfied. This concludes the proof of the feasibility of $(x,u,v)$.

In the sequel, we find $\bar\delta$ and $\bar\eps$ so that, if $W$ satisfies \eqref{weak_optimality_W}, then the transformed $w$ verifies \eqref{ineqw} for the given $\eps.$ The analysis is analogous for both controls $u$ and $v$, hence we will conduct the calculations only for $u$. We have
{\small
\begin{equation}
	\begin{array}{cc}
		\disp \int_{I_k\cap \Ih_k} |u_i(t) - \uh_i(t)|\dd t & \disp \le
		\int_{I_k\cap \Ih_k} \left| u_i^k\left(\frac{t - T_{k-1}}{T_k - T_{k-1}}\right) - \uh_i^k\left(\frac{t - T_{k-1}}{T_k - T_{k-1}}\right)\right|\dd t \\
		
		& \disp + \int_{I_k\cap \Ih_k} \left| \uh_i^k\left(\frac{t - T_{k-1}}{T_k - T_{k-1}}\right) - \uh_i^k\left(\frac{t - \Th_{k-1}}{\Th_k - \Th_{k-1}}\right)\right|\dd t.
	\end{array}
\end{equation}}
The first integral in the r.h.s. of latter display is bounded by $\bar\varepsilon|I_k\cap \Ih_k|$ in view of \eqref{weak_optimality_W}. For the second term, recall that $\uh$ is continuous on $[0,T]$ and so are the components of $\uh^k$ over $\Ih_k$, so that they are uniformly continuous over these intervals. Therefore, for each $k = 1, \cdots, N,$ we can find some $\bar\delta_k>0$ such that, if $|T_k - \Th_k| < \bar\delta_k,$ then
\begin{equation*}
\left| \uh_i^k\left(\frac{t - T_{k-1}}{T_k - T_{k-1}}\right) - \uh_i^k\left(\frac{t - \Th_{k-1}}{\Th_k - \Th_{k-1}}\right) \right| < \bar\varepsilon
\end{equation*}
for every component of $\uh^k$. Hence we only need to choose $\disp \bar\delta:=\min_{k = 1, \cdots, N} \bar\delta_k$. We proved that
\be
\label{estimate_u_1}
\disp \int_{I_k\cap \Ih_k} |u_i(t) - \uh_i(t)|\dd t  \leq 2\bar\varepsilon|I_k\cap \Ih_k|.
\ee
Next, we need to estimate the integral outside the intersection $I_k\cap \Ih_k.$ We assume w.l.o.g. that $T_k < \Th_k$ hence, in view of \eqref{weak_optimality_W},
\begin{equation}
\label{estimate_u_2}
\int_{T_k}^{\Th_k}|u_i(t) - \uh_i(t)|\dd t \le \bar\delta\bar\varepsilon.
\end{equation}
Adding up all the terms, we get from \eqref{estimate_u_1}-\eqref{estimate_u_2}, that
$$\norm{u_i - \uh_i}_1 < \bar\varepsilon(2T + (N-1)\bar\delta).
$$ 
An analogous estimate can be obtained for $\|v-\vh\|_1.$ Finally, taking into account all the control components  $m$ of the linear controls and $l$ from the nonlinear controls, we get that, if
\begin{equation*}
\bar\varepsilon(2T + (N-1)\bar\delta) < \frac{\varepsilon}{m + l},
\end{equation*}
then $\norm{u - \uh}_1 < \varepsilon$, as desired.
\end{proof}

}\fi

\subsection{The shooting algorithm for the transformed problem} In order to have an algorithm suitable to solve control constrained problems, our final step is to define a proper shooting function and apply the procedure described in Section \ref{shooting_algorithm}.

We start by stating the PMP for this unconstrained problem (TP).
Define the endpoint Lagrangian
\begin{equation}
	\tilde{\ell}:= \phi(x^1(0), x^N(1)) + \sum_{j = 1}^{d_{\eta}}\beta_j\eta_j(x^1(0), x^N(T)) + \sum_{k = 1}^{N-1}\theta^k\left(x^k(1)-x^{k+1}(0)\right).
\end{equation}
Note that each multiplier $\beta_j$ is associated with the endpoint constraints that come from the original problem and each $\theta^k$ is associated with the additional constraints of continuity of the state from (TP). The pre-Hamiltonian of (TP) is given by
\begin{equation*}
\begin{split}
	\tilde H &:= \sum_{k = 1}^{N}(T_k - T_{k-1})H^k, 
\end{split}
\end{equation*}
where {\small $H^k := \disp p^k\cdot \left(\sum_{i\in B_k\cup \{0\}}f_i(x^k,u^k) + \sum_{i\in S_k} v^k_if_i(x^k,u^k)\right).$}
Hence, from the PMP, the costates follow the dynamics
\begin{equation}
	\dot{p}^k = - (T_k - T_{k-1})D_{x^k}H^k,
\end{equation}
with transversality conditions
\begin{gather}
	p^1(0) = -D_{x^1_0}\phi - \sum_{j -1}^{d_{\eta}}\beta_jD_{x^1_0}\eta_j(x^1(0), x^N(T)),\\
	\label{costates_TP_continuity}
	\begin{array}{ll}
		p^k(1) = \theta^k, & \text{ for $k = 1, \cdots, N-1,$}\\
		p^k(0) = \theta^{k-1}, & \text{ for $k = 2, \cdots, N,$}
	\end{array}\\
	p^N(1) = D_{x_1^N}\phi + \sum_{j -1}^{d_{\eta}}\beta_jD_{x_1^N}\eta_j(x^1(0), x^N(T)).
\end{gather}
Note that equation \eqref{costates_TP_continuity} can be replaced by
\begin{equation}
	\label{costates_TP_continuity2}
	p^k(1) = p^{k+1}(0), \quad\text{for $k = 1, \cdots, N-1$},
\end{equation}
hence, eliminating the multipliers $\theta^k$.
We must also address the costates $p^{T_k}$ associated with the switching times, which satisfy
{\small 
\begin{equation}
\label{costatesTk}
	\begin{array}{cccc}
		\dot{p}^{T_k} = -H^k + H^{k+1}, & p^{T_k}(0) = 0,& p^{T_k}(1) = 0, & \text{ for $k = 1, \cdots, N-1$}. 
	\end{array}
\end{equation}}
Combining all conditions from \eqref{costatesTk}, we obtain
\begin{equation}
	\label{condition.pTk}
	\int_{0}^{1}\left(H^{k+1} -H^k \right)\dd t = p^{T_k}(0) - p^{T_k}(1) = 0.
\end{equation}
Since the dynamics are autonomous, the Hamiltonian is constant for the optimal trajectory and we equivalently express the conditions \eqref{condition.pTk} for $p^{T_k}$ as 
\begin{equation}
	H^k = H^{k+1}, \quad \text{ for $k = 1, \cdots, N-1.$}
\end{equation}
Now we are in position to adapt the shooting scheme for solving (TP). Following the steps from Section \ref{shooting_algorithm} we start by finding the feedback form for the controls. It suffices to use the representation given in equation \eqref{control_feedback_mappings} 
\begin{equation}
	u^k = U\left(x^k, p^k\right), \quad v^k = V_{S_k}\left(x^k, p^k\right),\quad \text{ for $k = 0, \cdots, N.$}
\end{equation}
By Lemma \ref{CP-TP-relation} such controls must also be feasible for (TP) and when the feedback arguments $\xh^k$ and $\ph^k$ correspond to the nominal trajectory, we obtain optimal controls.

We must also define an appropriate shooting function that will express the stationarity of the Hamiltonian, the initial-final constraints and transversality conditions. Stationarity with respect to the nonlinear controls is equivalent to the feedback representation for $u$ given in equation \eqref{control_feedback_mappings}. For the linear controls, the feedback form is equivalent to $\ddot{H}_{v_{S_k}} = 0$. Hence we must also impose $H^k_{v^k_i}(0) = 0$ and $\dot{H}^k_{v^k_i}(0) = 0$ to ensure the stationarity $H_{v_k} = 0$.

Note that we can choose to include the constraints related the continuity of the states and costates or integrate each $x^k$ and $p^k$ using the final values of $x^{k-1}$ and $p^{k-1}$ as initial conditions. The clear advantage of the latter strategy is the smaller number of shooting variables, i.e. the initial conditions for states and costates at the switching times can be omitted. On the other hand, explicitly including these constraints makes the algorithm more stable numerically and favors the parallelization of computational implementations, see \cite{stoer2013introduction}.

The following is the shooting function associated to (TP) with the full set of shooting variables
\begin{equation}
\label{shooting.function.TP}
	\begin{array}{c}
		\mathcal{S} \colon \cR^{Nn}\times\cR^{Nn,*}\times\mathbb{R}^{N-1}\times\mathbb{R}^{d_{\eta}} \to  \cR^{(N-1)n + d_{\eta}}\times \cR^{(N+1)n + N-1 +2\sum |S_k|,*}\\
		\nu \mapsto \mathcal{S}(\nu) :=
		\left(
		\begin{array}{c}
			\eta(x^1(0), x^N(1)) \\
			\left( x^k(1) - x^{k+1}(0)\right)_{k = 1,\cdots,N-1}\\
			\left( p^k(1) - p^{k+1}(0)\right)_{k = 1,\cdots,N-1}\\
			p^1(0) + D_{x_0^1}\tilde \ell\left(x^1(0), x^N(1)\right)\\
			p^N(1) - D_{x_1^N}\tilde \ell\left(x^1(0), x^N(1)\right)\\
			\left( H^k(1) - H^{k+1}(0)\right)_{k = 1,\cdots,N-1}\\
			\left( p^i\cdot f_i\left(x^i, U^i\right)(0) \right)_{\tiny
			\begin{array}{l}
				\tiny i \in S_k\\
				\tiny k = 1,\dots, N
			\end{array}
			}\\
			\disp
			\left(p^i\cdot [f_0, f_i]^x(0)\right)_{\tiny
				\begin{array}{l}
				\tiny i \in S_k\\
				\tiny k = 1,\dots, N
				\end{array}
			}
		\end{array}
		\right)
	\end{array}
\end{equation}
where we define the vector of shooting arguments as
{\small 
\begin{equation}
	\nu := \left(\left(x^k(0)\right)_{k=1}^{N}, \left(p^k(0)\right)_{k=1}^{N}, \left(T_k\right)_{k=1}^{N-1}, \beta\right).
\end{equation}}
We recall equation \eqref{switching.function.D1} that gives a concise analytical form for $\dot{H}_{v_i}$ and was used in the formulation of the last component of the above shooting function. 

Since the new problem (TP) falls in the same category of unconstrained problem (OC), we join Lemma \ref{CP-TP-relation} and Theorem \ref{convergence.unconstrained-problem} in the following result. 

\begin{theorem}
	\label{convergence.constrained_problem}
	If $\wh$ is a Pontryagin minimum of \textnormal{(CP)} such that $\Wh$ given in \eqref{Wh} satisfies the coercivity condition \eqref{Omega.P2.coersity.unique} for problem \textnormal{(TP)}, then the shooting algorithm for \textnormal{(TP)} is locally quadratically convergent.
\end{theorem}

\section{Examples}
\label{Section_examples}
\subsection{Degenerate Linear Quadratic Problem}
\label{degen_lin_quad}
In this section we discuss the application of the shooting algorithm to a toy problem. We consider the following partially-affine problem, inspired by the examples in \cite{dmitruk2010analysis,aronna2013convergence}.
\begin{equation}
\label{degen_lin_quad_prob}
\begin{array}{cl}
	\text{minimize}   & \disp -2x_2(2) + \int_{0}^{2}\left( x_1^2 + x_2^2 + u^2 + 10x_2v\right)\dd t \\
	\text{subject to} & \dot{x}_1 = x_2 + u,\\
					  &\dot{x}_2 = v, \\
					  & 0\le v(t) \le 0.5, \quad \text{a.e. on $[0,T]$},\\
					  & x_1(T) = 1,\\
					  & x_1(0) = x_2(0) = 0.
\end{array}
\end{equation}
We start by obtaining an estimate for the optimal control structure. This was done by using the BOCOP package \cite{bonnans2012bocop}, where we found that the optimal solution presents a {\em bang-singular-bang} structure.

Figure \ref{part_affine_all} shows a comparison of the solutions of our shooting algorithm and the one obtained from BOCOP. The latter already shows a good approximation of the singular control, however it has poorer performance around the switching times. Another interesting numerical phenomenon usually observed in direct methods, is the fact that the control variables have a tendency to have a slower convergence than the state and costate variables, see \cite{hager2000runge}. This can be perceived in the comparison graph of the nonlinear controls (see also Figure \ref{part_affine_all}). Since the shooting algorithm uses the analytical expression of the optimal controls in feedback form, we can expect more accurate results. 
\if{
\begin{figure}[h!] 
	\indent{\centering\includegraphics[width=1.2\linewidth, center]{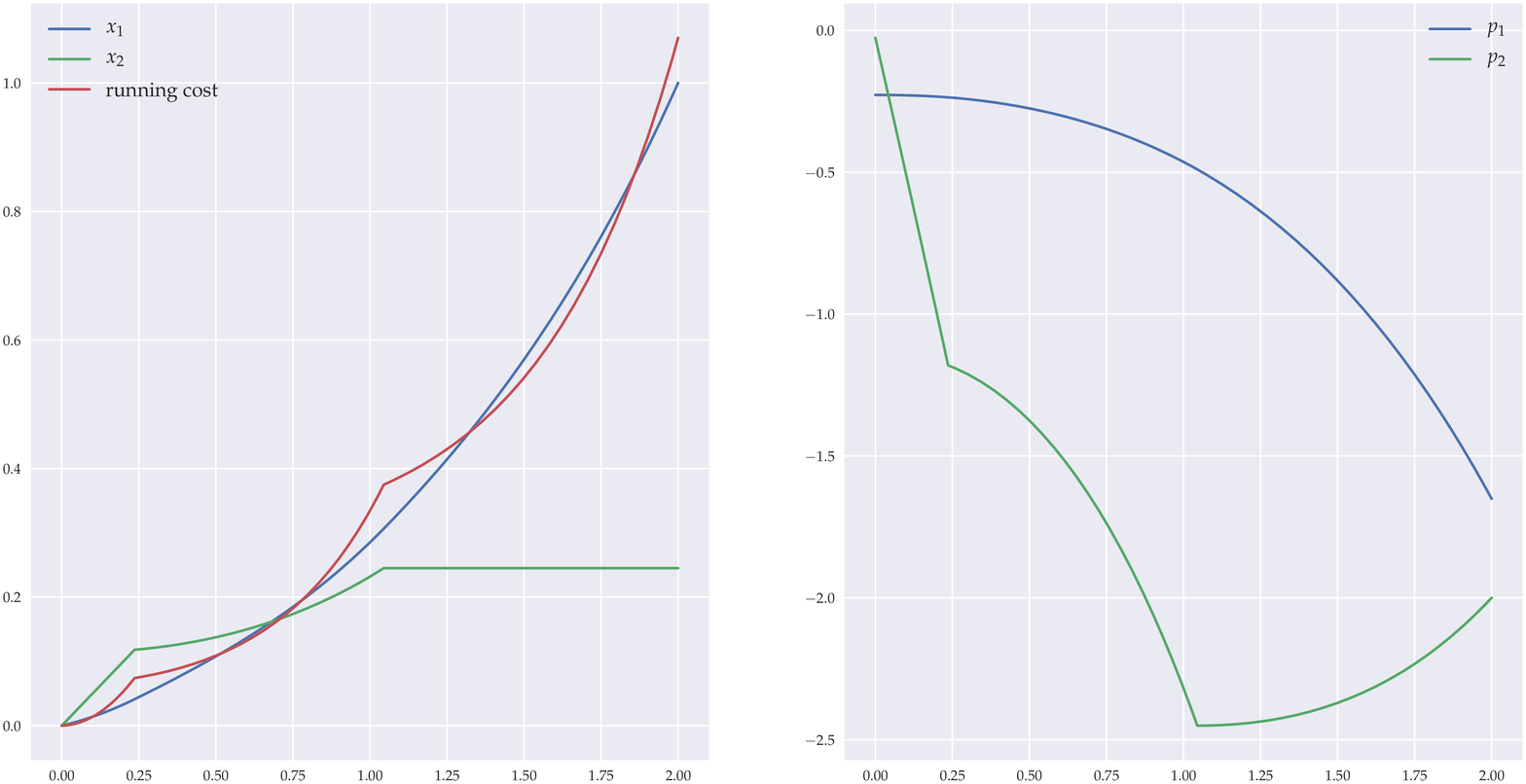}}\\
	\indent 
	\caption[]{}
	\label{part_affine_states}
\end{figure} 
\begin{figure}[h!] 
	\indent{\centering\includegraphics[width=1.4\linewidth, center]{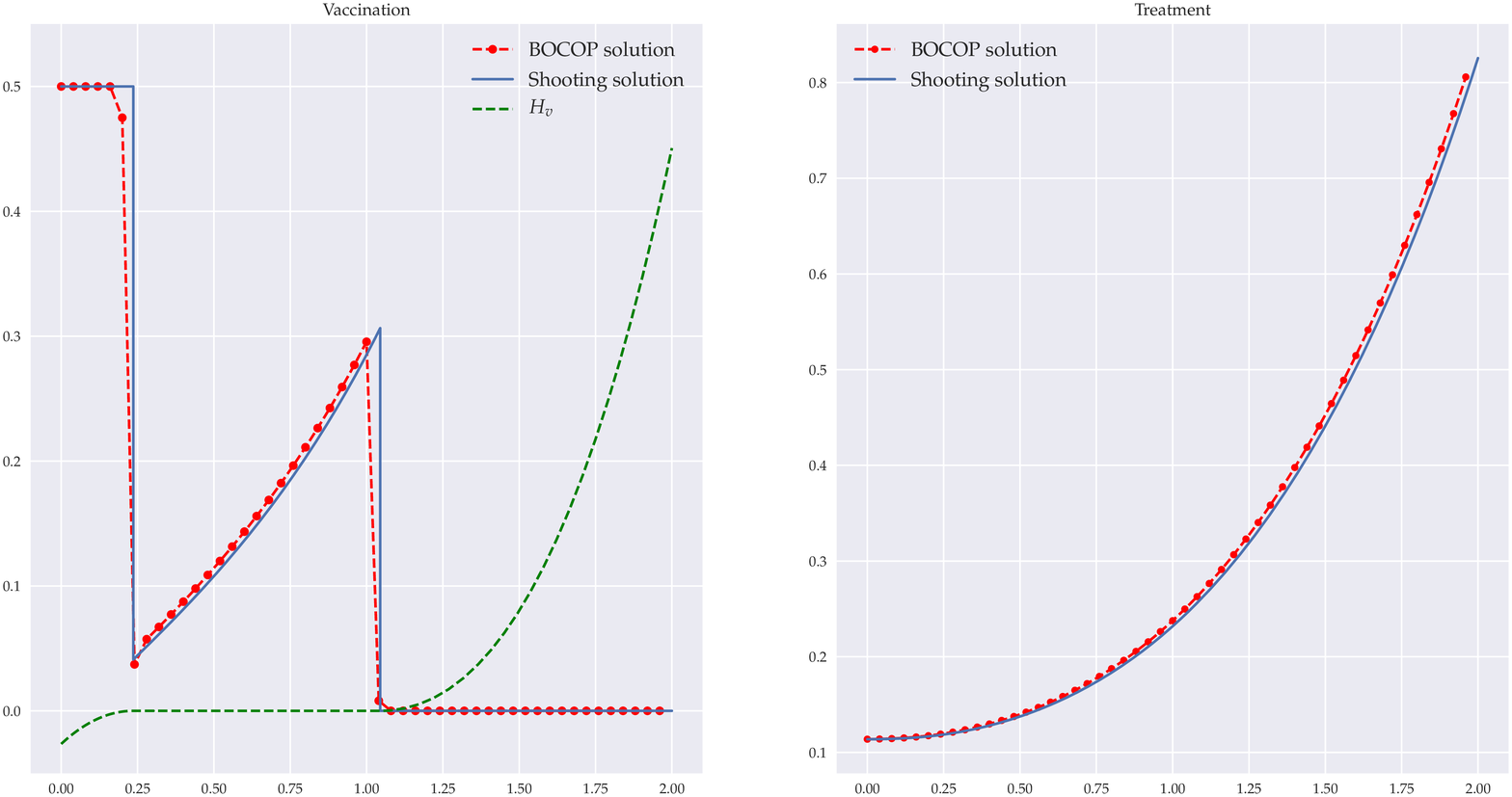}}\\
	\indent 
	\caption[]{}
	\label{part_affine_controls}
\end{figure} 
}\fi

\begin{figure}[h] 
	\indent{\centering\includegraphics[width=1.1\linewidth, center]{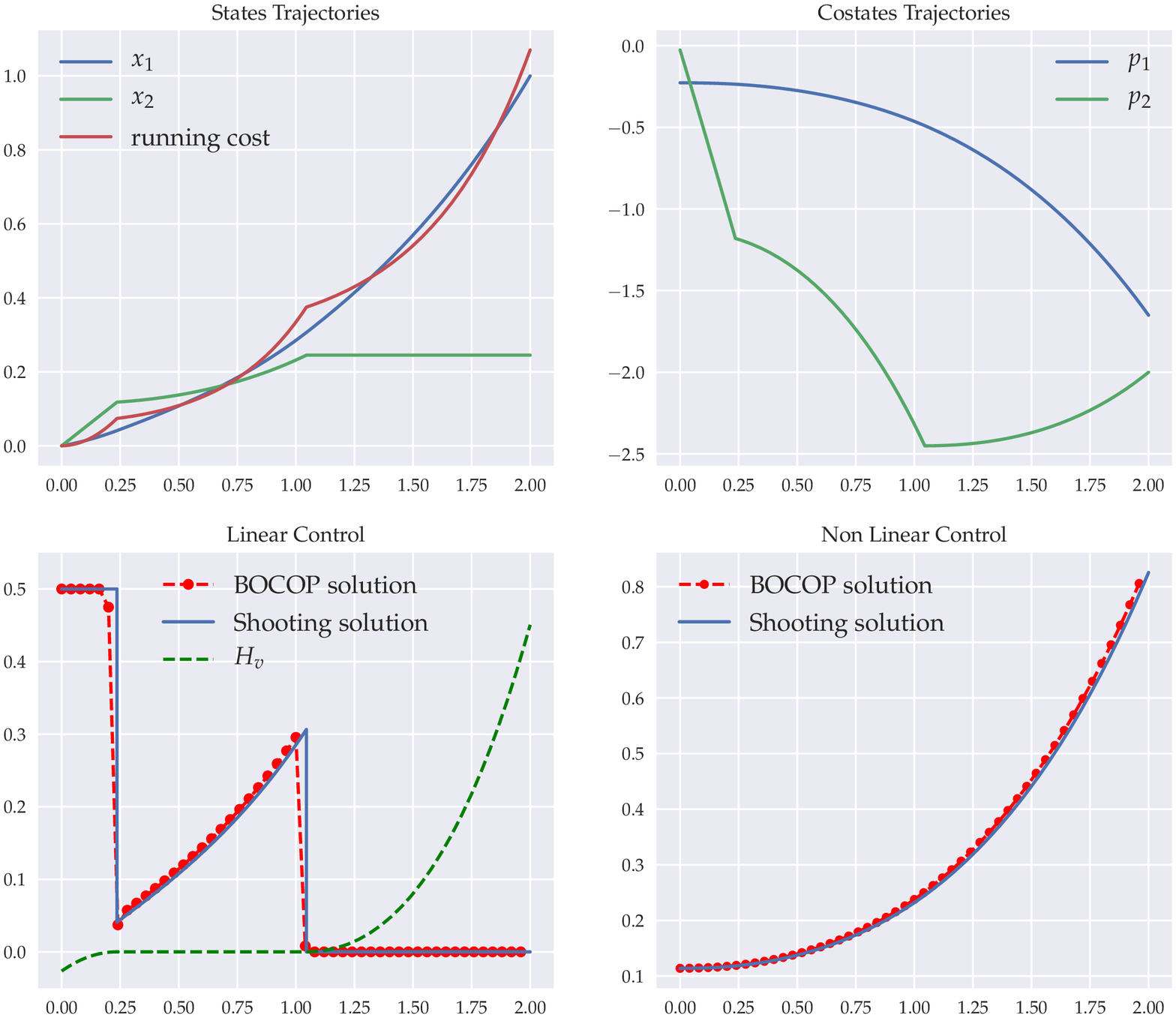}}\\
	\indent 
	\caption[]{Optimal trajectories and controls for problem \eqref{degen_lin_quad_prob}}
	\label{part_affine_all}
\end{figure}

\subsection{Optimal Control of an SIRS Epidemiological Model}
\label{SIRS_OC}
In this section we follow \cite{gaff2009optimal, ledzewicz2011optimal} where the authors discuss problems regarding the optimal control of various SIR (susceptible-infected-recovered) models used to describe the spread of an epidemic in some demographic population. The control is performed either through a term $v$ representing vaccination of susceptible individuals $S$, leading them to the recovered, and temporarily immune, class $R$; or through the treatment of infected individuals that is represented by a second control variable $u$, taking individuals from the infected compartment $I$ to $R$. In this article, we consider the variation known as SIRS model, which takes into account the effect of temporary immunity of recovered individuals $R$, gradually reintroducing them into the susceptible class $S$. This brief discussion is encapsulated in the system below, the description of the involved parameters being given in Table \ref{Parameter table}:
\begin{equation}
\label{NSI}
\begin{array}{rl}
\dot{N} &= F(N) - \delta I - \mu N, \\
\dot{S} &= F(N) - \beta \frac{IS}{N} -vS + \omega (N - S - I) - \mu S, \\
\dot{I} &= \beta \frac{IS}{N} - (\gamma + \delta + u)I - \mu I.
\end{array}
\end{equation}
Here $N$ represents the total number of individuals in the population, $N = S + I + R$, the function $F:[0,K) \to \mathbb{R}_+$ is the population growth function assumed to be logistic of the form $F(N) = \alpha N(1 - N/K)$.

\begin{table}[!h]
	\centering
	\begin{tabular}{cll}
		\hline
		{\bf Parameter} & {\bf Biological Meaning} & {\bf Values} \\
		\hline
		\hline
		$N_0$    & initial total population       & 5000 humans    \\
		$S_0$    & initial susceptible population & 4500 humans    \\
		$I_0$    & initial infected population    & 499  humans    \\
		$\alpha$ & population growth rate         & $4 \times 10^{-5}$  days$^{-1}$  \\
		$K$      & carrying capacity              & 5000           \\
		$\mu$    & natural death rate of population & $10^{-5}$  days$^{-1}$  \\
		$\beta$  & incidence rate                 & 0.5 days$^{-1}$      \\
		$\omega$ & waning rate                    & 0.01 days$^{-1}$     \\
		$\gamma$ & recovery rate                  & 0.1 days$^{-1}$      \\ 
		$\delta$ & death rate due to disease      & 0.1 days$^{-1}$      \\
		$B_1$    & cost per infection             & 1              \\
		$B_2$    & cost per vaccination           & 50             \\
		$B_3$    & cost per treatment             & 1000           \\
		$v_{\max}$& maximum vaccination rate  	  & 0.25       \\
		$T$      & horizon of analysis   		  & 100 days       \\
		\hline
	\end{tabular}
	\caption{Biologically feasible parameters.}
	\label{Parameter table}
\end{table}

Our goal is to minimize the amount of ill individuals with the lowest cost of vaccination and treatment over a time window, hence we choose the cost function
\begin{equation}
	C(t) := \label{cost_function_SIR}
	\int_{0}^{t} \left(B_1 I(s) + B_2v(s) + B_3u^2(s) \right)\dd s 
\end{equation}
The choice of the terms $B_1I$ and $B_3u^2$ follows \cite{gaff2009optimal}. When compared to treatment policies, vaccination is more easily implemented and hence appears we choose to make it appear linearly in the cost, as done in \cite{ledzewicz2011optimal}. The linear dependence on the vaccination might result in bang-bang optimal controls as in \cite{behncke2000optimal}, however, the parameters values in Table \ref{Parameter table} were chosen to favor the appearance of singular arcs among realistic parameters given in \cite{gaff2009optimal}.

Now we introduce the optimal control problem in Mayer form
\begin{equation}
\label{SIR_problem}
\begin{array}{cl}
\text{minimize} & C(T) \\
\text{subject to} & \eqref{NSI},\\
& 0\le v(t) \le v_{\max},\,\,\, 0\le u(t) \quad \text{a.e. on $[0,T]$}\\
& N(0) = N_0, \quad S(0) = S_0, \quad I(0) = I_0, \quad C(0) = 0.
\end{array}
\end{equation}
We show below that the restriction of non negativity on the nonlinear control $u$ is redundant, since a ``negative treatment" is never optimal. 
\begin{proposition}
	If $(\uh,\vh)$ is optimal for \eqref{SIR_problem}, then $\uh \geq 0 $ a.e. on $[0,T].$
	\label{non_negative_treatment}
\end{proposition}
\begin{proof}
	Suppose an optimal solution $(\uh, \vh)$ is such that $\uh$ presents negative values in a set of positive measure. Define a new control strategy, where $\vh$ remains unchanged and exchange $\uh$ by $\ut := \max\{\uh, 0\}$. The cost associated with $\vh$ is unaffected, the term depending on the treatment, $\int_{0}^{T}u^2(s)\dd s$, is clearly less expensive for $\ut$ and it remains to be checked the influence on the cost associated with the amount of infected individuals of this strategy. 
	
	With this in mind, let $(N,S,I)$ and $(\tilde N, \St, \tilde I)$ be the solutions for \eqref{NSI} with the control strategies $(\uh, \vh)$ and $(\ut, \vh)$, respectively. To conclude our argument, it suffices to show that the quantity $z:= \tilde I - I$ is non positive. Note that
	\begin{align*}
	\dot{z} = \dot{\tilde I} - \dot{I} = \beta \tilde I \frac{\St}{\tilde N} - \beta I \frac{S}{N} - (\gamma + \delta + \mu)(\tilde I - I) + \uh I\chi_{\{\uh < 0\}}.
	\end{align*}
	Hence, we can define a continuous function $c(t)$, depending on $\St, \tilde N, S, N$, such that
	\begin{equation*}
	\dot{z} \le \left(\beta c(t) - \left(\gamma + \delta + \mu + \uh \xi_{\{\uh >0\}} \right)\right)z + \uh I\chi_{\{\uh < 0\} }.
	\end{equation*}
	Setting $a(t) := \beta c(t) - \left(\gamma + \delta + \mu + \uh \chi_{\{\uh >0\}} \right)$ and $b(t) := \uh I\chi_{\{\uh < 0\} }$, by Gronwall's lemma, we have that
	\begin{equation*}
	z(t) \le z(0)\exp\left(\int_{0}^t a(s) \dd s\right) + \int_{0}^{t}b(s)\exp\left(\int_{0}^s a(\sigma) \dd \sigma\right)\dd s.
	\end{equation*}
	By definition, $z(0) = 0$ and $b \le 0$, thus $z\le 0$, this is $\tilde I \le I$. 
\end{proof}
With the aid of the previous Proposition \ref{non_negative_treatment}, our control problem \eqref{SIR_problem} satisfies all assumptions from Section \ref{simpler_control_constraints}, since the constraint $u\ge 0$ can be removed, and we can apply our algorithm. The singular vaccination strategies are obtained using the expression for $\ddot{H}_v$ derived in \eqref{kernel_equation.prev1}. The complete analytical computation can be found in Appendix \ref{appendix}, however, our computational implementation relies on SymEngine - a {\em Computer Algebra System} (CAS), see \cite{certiksymengine} - that automates this laborious task and other computations necessary to formulate our algorithm.

As done for the previous example, we used BOCOP \cite{bonnans2012bocop} to get an estimate of the shooting parameters and switching times in order to initialize our algorithm. The results are shown in Figures \ref{trajectories} and \ref{controls}.

\begin{figure}[h!]
	\centering
	\includegraphics[width=1.4\linewidth, center]{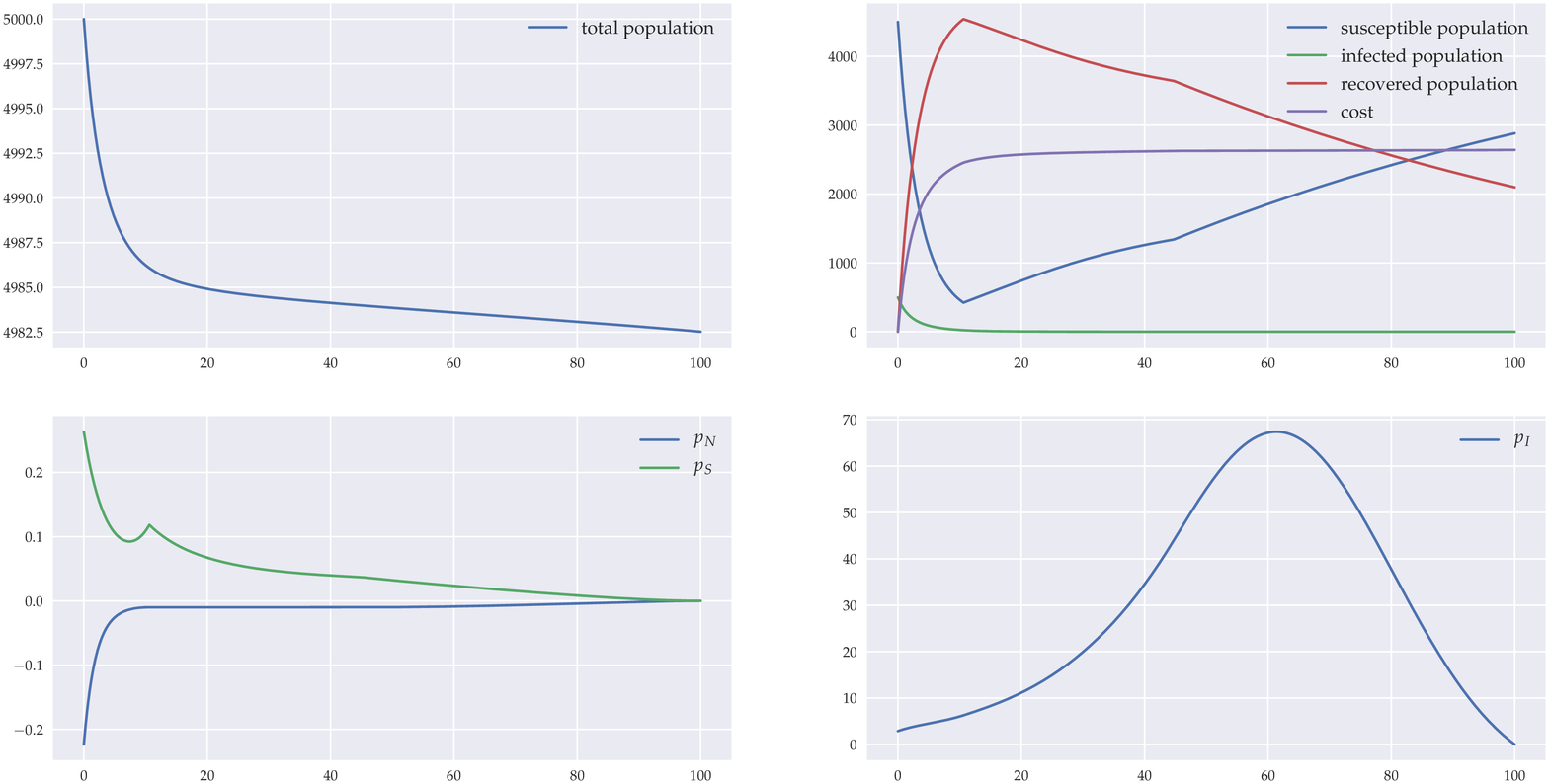}\\
	\indent 
	\caption[]{Optimal trajectories for problem \eqref{SIR_problem}.}
	\label{trajectories}
\end{figure} 

\begin{figure}[h!] 
	\indent{\centering\includegraphics[width=1.4\linewidth, center]{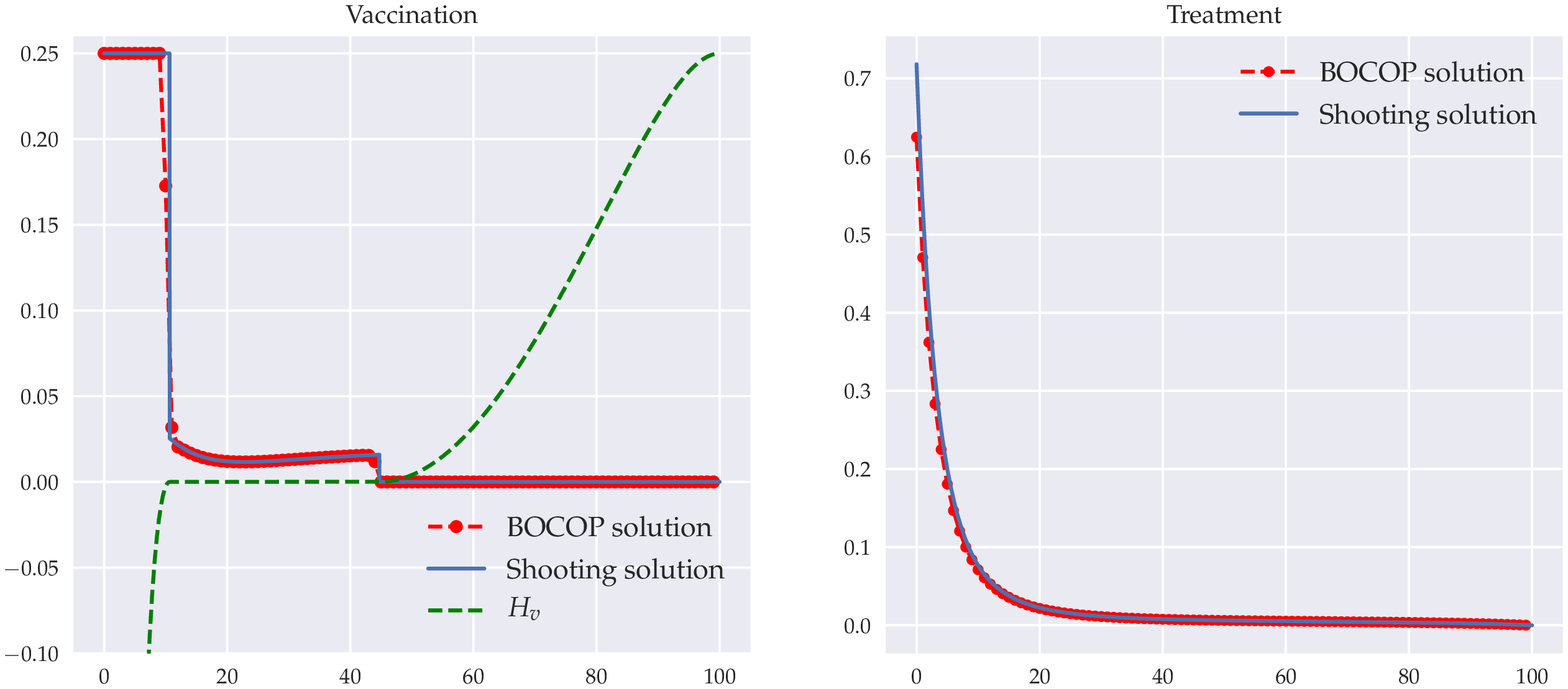}}\\
	\indent 
	\caption[]{Optimal controls for problem \eqref{SIR_problem}.}
	\label{controls}
\end{figure}

\section{Conclusion}
\label{conclusion}
In this article we have studied the shooting algorithm for partially-affine optimal control problems, this is, problems where some control components appear linearly and others non linearly in the Hamiltonian. Many of the results here discussed are extensions of previous works concerning the totally-affine case. Such extensions were only possible after the development of no-gap second order necessary and sufficient conditions for weak optimality (given in \cite{Aronna2018}). We have also revised second order analysis results that enabled us to provide a more detailed characterization of singular controls. Additionally, we were able to relate the mentioned sufficient conditions to the well-known strengthened generalized Legendre-Clebsch conditions. Concerning the implementation of the shooting algorithm, we were able to automate lengthy and tedious computations necessary for its formulation.

The case with control constraints is also tackled, by means of a transformation that reduces this case to the unconstrained one. Both this transformation and the computation of singular controls are automated in our implementation, which is demonstrated by two numerical examples. The first one was chosen to illustrate how even simple problems can become fairly large once we introduce the associated transformed problem. The second example discusses the optimal control problem of an SIRS epidemiological model with vaccination and treatment acting as controls. This problem requires lengthy computations to obtain the analytical expressions of the singular arcs and serves as a proof of usefulness of our automated implementation. 

\appendix

\section{Computation of Singular Vaccination Strategies}
\label{appendix}
In this appendix we develop the computations of the singular vaccination strategies from Section \ref{SIRS_OC} in full detail. To shorten notation, we define the state vector $x := (N, S, I, C)^T$ and rewrite the dynamics as
\begin{equation}
	\begin{array}{cc}
		\dot{x} &= f_0(x,u) + vf_1(x)\\
	\end{array}
\end{equation}
where 
\begin{equation}
	f_0(x, u) = \left(
	\begin{array}{l}
		F(N) - \delta I - \mu N\\
		F(N) - \beta\frac{IS}{N} + \omega R - \mu S\\
		\beta\frac{IS}{N} - (\delta + \gamma + u + \mu)I\\
		B_1I + B_3 u^2
	\end{array}
	\right), \quad 
	f_1(x) = \left(
	\begin{array}{c}
		0\\
		-S\\
		0\\
		B_2
	\end{array}
	\right)
\end{equation}

Following the arguments from Section \ref{differential_algebraic}, the singular arcs for the linear control, {\em i.e.} vaccination, satisfy the following expression
\begin{equation}
	\gamma_{01} + v_{\rm sing}\gamma_{11} = 0.
\end{equation}
Let us compute the quantities $\gamma_{01}$ and $\gamma_{11}$, as defined in \eqref{kernel_equation.prev1}. Initially note that 
{\small 
\begin{gather*}
	Df_0 = \left(
	\begin{array}{cccc}
		F'(N) & 0 & -\delta & 0\\
		F'(N) + \omega + \beta\frac{SI}{N^2} &  -\omega - \beta \frac{I}{N} & -\omega -\beta\frac{S}{N} & 0\\
		-\beta\frac{SI}{N^2} & \frac{I}{N} & \beta\frac{S}{N} - (\delta + \gamma + u + \mu) & 0\\
		0 & 0 & B_1 & 0
	\end{array}
	\right), \\
	Df_1 = \left(
	\begin{array}{cccc}
		0 & 0 & 0 & 0\\
		0 & -1 & 0 & 0\\
		0 & 0 & 0 & 0\\
		0 & 0 & 0 & 0
	\end{array}
	\right).
\end{gather*}}
In order to compute $\gamma_{01}$ and $\gamma_{11}$, we start with the Lie bracket $[f_0, f_1]$:
\begin{align*}
	[f_0, f_1] = Df_1f_0 - Df_0f_1
	=
	\left(
	\begin{array}{c}
		0\\
		-\left(F(N) + \omega(N-I)\right)\\
		\beta \frac{SI}{N}\\
		0
	\end{array}
	\right).
\end{align*}
Notice that $[f_0,f_1]$ does not depend on the nonlinear control $u$, hence the expressions for $\gamma_{01}$ and $\gamma_{11}$ become $p\cdot[f_0, [f_0, f_1]]$ and $p\cdot[f_1, [f_0, f_1]]$, respectively. For $p\cdot[f_1, [f_0, f_1]]$ we have
{\small
\begin{align*}
	[f_1, [f_0, f_1]] & = D[f_0, f_1]f_1 - Df_1[f_0, f_1]\\
	& 
	=\left(
	\begin{array}{cccc}
		0 & 0 & 0 & 0 \\
		-F'(N) - \omega & 0 & \omega & 0\\
		-\beta\frac{SI}{N^2} & \beta\frac{I}{N} & \beta\frac{S}{N} & 0\\
		0 & 0 & 0& 0
	\end{array}
	\right)
	\left(
	\begin{array}{c}
		0\\
		-S\\
		0\\
		B_2
	\end{array}
	\right)
	\\
	&\quad- \left(
	\begin{array}{cccc}
		0 & 0 & 0 & 0\\
		0 & -1 & 0 & 0\\
		0 & 0 & 0 & 0\\
		0 & 0 & 0 & 0
	\end{array}
	\right)
	\left(
	\begin{array}{c}
		0\\
		-\left(F(N) + \omega(N-I)\right)\\
		\beta \frac{SI}{N}\\
		0
	\end{array}
	\right)\\
	& 
	=\left(
	\begin{array}{c}
		0\\
		0\\
		-\beta\frac{SI}{N}\\
		0
	\end{array}
	\right)
	-
	\left(
	\begin{array}{c}
		0\\
		F(N) + \omega(N-I)\\
		0\\
		0
	\end{array}
	\right) = -2\left(
	\begin{array}{c}
	0\\
	0\\
	\beta\frac{SI}{N}\\
	0
	\end{array}
	\right) + [f_0, f_1].
\end{align*}}
Using the Goh conditions \eqref{Goh.condition} and the fact that $\dot{H}_v = 0$, we obtain
\begin{equation}
	\gamma_{11} = p\cdot[f_1, [f_0, f_1]] = -2\beta\frac{SIp_I}{N}.
\end{equation}
Moving on to $[f_0, [f_0, f_1]]$, after some algebraic simplifications, we have
{\small
\begin{align*}
	&[f_0, [f_0, f_1]] = D[f_0, f_1]f_0 - Df_0[f_0, f_1]\\
	&= \left(
	\begin{array}{cccc}
	0 & 0 & 0 & 0 \\
	-F'(N) - \omega & 0 & \omega & 0\\
	-\beta\frac{SI}{N^2} & \beta\frac{I}{N} & \beta\frac{S}{N} & 0\\
	0 & 0 & 0& 0
	\end{array}
	\right)
	\left(
	\begin{array}{l}
	F(N) - \delta I  - \mu N\\
	F(N) - \beta\frac{IS}{N} + \omega R - \mu S\\
	\beta\frac{IS}{N} - (\delta + \gamma + u + \mu)I\\
	B_1I + B_3\tau^2
	\end{array}
	\right)\\
	& - \left(
	\begin{array}{cccc}
	F'(N) & 0 & -\delta & 0\\
	F'(N) + \omega + \beta\frac{SI}{N^2} &  -\omega - \beta \frac{I}{N} & -\omega -\beta\frac{S}{N} & 0\\
	-\beta\frac{SI}{N^2} & \frac{I}{N} & \beta\frac{S}{N} - (\delta + \gamma + u + \mu) & 0\\
	0 & 0 & B_1 & 0
	\end{array}
	\right)
	[f_0,f_1]\\
	\if{
	&= \left(
	\begin{array}{c}
		0\\
		-(F'(N) + \omega)(F(N) - \delta I) + \omega \left(\beta\frac{SI}{N} - (\delta + \gamma + u) I\right)\\
		-\beta\frac{SI}{N^2}(F(N) - \delta I) + \beta\frac{I}{N}(F(N) - \omega (N-I)) + {\color{red}\beta\frac{SI}{N}( -\beta\frac{I}{N} - \omega)} + {\color{blue}\beta \frac{S}{N}\left(\beta\frac{SI}{N} - (\delta + \gamma + u + \mu) I\right)}\\
		0
	\end{array}
	\right) \\& \quad-
	\left(
	\begin{array}{c}
		-\delta\beta \frac{SI}{N}\\
		{\color{red} \left(\omega +\beta \frac{I}{N}\right)(F(N) + \omega(N-I))} - \left(\omega + \beta \frac{S}{N}\right)\beta\frac{SI}{N}\\
		-(F(N)+\omega(N-I))\beta\frac{I}{N} + {\color{blue}\beta\frac{S}{N}\left(\beta\frac{SI}{N} - (\delta + \gamma + u + \mu)I\right)}\\
		B_1\beta\frac{SI}{N}
	\end{array}
	\right)\\ }\fi
	&= \beta\frac{SI}{N}w_1 +  w_2+\left(\beta\frac{I}{N} + \omega\right)[f_1, [f_0, f_1]],
\end{align*}}
where the vectors $w_1$ and $w_2$ are given by
{\small 
\begin{gather*}
	w_1 := \left(
	\begin{array}{c}
	\delta\\
	2\omega + \beta\frac{S}{N}\\
	-(F(N) - \delta I)/N + 2\beta\frac{I}{N}\frac{p_I}{p_S}\\
	-B_1
	\end{array}
	\right), \\
	w_2 :=
	\left(
	\begin{array}{c}
	0\\
	-(F'(N) + \omega)(F(N) - \delta I) - \omega I(\delta + \gamma + u)\\
	0\\
	0
	\end{array}
	\right).
\end{gather*}
}
Notice that the appearance of the term $[f_1, [f_0, f_1]]$ simplifies the final expression of the singular controls since this term cancels out with the denominator $\gamma_{11}$. Hence, the expression for the singular control becomes
\begin{align}
	v_{\rm sing} = - \frac{\gamma_{01}}{\gamma_{11}} = -\left(\omega + \beta\frac{I}{N}\right)+\frac{p}{2p_I}\cdot \left(w_1 + \frac{N}{\beta SI}w_2\right), 
\end{align}

\section{Proof of Technical Lemmas}
\label{appendix_Goh.computations}
\subsection{Proof of Lemma \ref{goh_computations}}
In this section we prove the following identity
\begin{equation}
\label{appendix_identity}
	E = -\frac{\partial \dot{H}_v}{\partial u} \quad \text{ and } - \frac{\partial \ddot{H}_v}{\partial v} = R - EH_{uu}^{-1}E^T,
\end{equation}
that are relevant in the recovery of the strengthened Legendre-Clebsch conditions \eqref{LC-like.mixedcontrols} from the sufficient conditions stated in Theorem \ref{SOSC}. Our strategy will be to establish the equality of the matrices involved entry wise. 

The first identity in \eqref{appendix_identity} is easily obtained with the definition of $E$ in \eqref{matrixM}. Before proceeding to the second one, let us establish some conventions that will make the computations clearer. Many conditions throughout the text state that some quantity $Q$ is null when evaluated along the optimal trajectories. For instance, we can recall the Goh conditions $\ph\cdot [f_i, f_j](\wh) = 0$. We want to stress out a distinction from the case that some other quantity $N$ identically assumes the value $0$, as is the case for $H_{vv} \equiv 0$. We will make a distinction of these two cases with the following notation
\begin{equation}
	Q = 0, \quad N \equiv 0.
\end{equation}
Naturally, if we take the time derivative of some quantity $Q = 0$, this property is maintained and we obtain $\dot{Q} = 0$. However, this is not true when we take partial derivatives, this is, $\partial_v Q$ is not necessarily null. With this in mind we recall the expressions from \eqref{kernel_equation.prev1} that were used to obtain the linear controls. While these expressions are suitable for this task, we cannot use them to compute the partial derivatives $\partial_v \ddot{H}_v$ since we have removed terms that vanish due to the Goh conditions in Proposition \eqref{Goh.condition} or as a consequence of the Legendre-Clebsch conditions \eqref{LegendreClebsch.mixed.equivalent}.

The full expressions we are interested in are still easily obtainable by using formula \eqref{genericvectorfield.D1}. We get,
\begin{align}
	\if{
	\label{Hv}
	H_{v_i} &= \ph\cdot f_i,\\
	\label{Hvdot}
	\dot{H}_{v_i} &= \ph\cdot[f_0, f_i] + \underbrace{ \sum_{k = 1}^m \vh_k\ph\cdot[f_k,f_i]}_{\textnormal{  $= 0,$ Goh  conditions}} + \underbrace{ H_{v_iu}\dot{\uh}, }_{\textnormal{  $= 0$, LC  conditions}}\\
	}\fi
	\label{Hvddot}
	\begin{split}
	\ddot{H}_{v_i} &=  
	\ph\cdot[f, [f_0, f_i]] + \ph\cdot D_u [f_0,f_i] \dot{\uh}\\
					&+ \sum_{k = 1}^m\left\{ \dot{\vh}_k\ph\cdot[f_k,f_i] + \vh_k \frac{\dd}{\dd t} \ph\cdot[f_k,f_i]\right\} + \frac{\dd}{\dd t}\left(H_{v_iu}\right)\dot{\uh} + H_{v_iu}\ddot{\uh}. 
	\end{split}
\end{align}
Notice that the coefficient of $\ddot{\uh}$ is zero, so we do not require further regularity for $\uh$. Taking the partial derivative w.r.t. $v_j$ in \eqref{Hvddot} yields
\begin{equation}
\label{DvddotHv}
\begin{split}
	&\frac{\partial \ddot{H}_{v_i}}{\partial v_j} =  \ph\cdot[f_j, [f_0,f_i]] + \ph\cdot D_u [f_0,f_i]\frac{\partial \dot{\uh}}{\partial v_j}\\
	& + \sum_{k = 1}^m\left\{ \frac{\partial \dot{\vh}_k}{\partial v_j}\underbrace{\ph\cdot[f_k,f_i]}_{ = 0} + \dot{\vh}_k \ph\cdot\underbrace{\frac{\partial }{\partial v_j}[f_k,f_i]}_{\equiv 0} + \frac{\partial \vh_k}{\partial v_j}\underbrace{\frac{\dd}{\dd t} \ph\cdot[f_k,f_i]}_{= 0}+ \vh_k \underbrace{\frac{\partial}{\partial v_j}\frac{\dd}{\dd t} \ph\cdot[f_k,f_i]}_{=: A_k}
	\right\}\\
	& + \underbrace{\frac{\partial}{\partial v_j}\frac{\dd}{\dd t}\left(H_{v_iu}\right)\dot{\uh}}_{=: B} + \underbrace{\frac{\dd}{\dd t}\left(H_{v_iu}\right)}_{= 0}\frac{\partial \dot{\uh}}{\partial v_j} + \underbrace{\frac{\partial}{\partial v_j}H_{v_iu}}_{\equiv 0}\ddot{\uh} + \underbrace{H_{v_iu}}_{= 0}\frac{\partial \ddot{\uh}}{\partial v_j}.
\end{split}
\end{equation}
Once again, the coefficients of $\dot{\vh}$ and $\ddot{\uh}$ vanish so we do not require any further regularity on the optimal controls. By computing the remaining time derivatives, we obtain the expressions
\begin{align}
	A_k = \frac{\partial}{\partial v_j}\frac{\dd}{\dd t} \ph\cdot[f_k,f_i]  &= \ph\cdot[f_j,[f_k,f_i]] + \ph\cdot D_u[f_k,f_i] \frac{\partial \dot{\uh}}{\partial v_j},\\
	B = \frac{\partial}{\partial v_j}\frac{\dd}{\dd t}\left(H_{v_iu}\right)\dot{\uh} & = \ph\cdot \left(
	\frac{\partial^2 f_i}{\partial x\partial u}f_j - \frac{\partial f_j}{\partial x}\frac{\partial f_i}{\partial u}
	\right)\dot{\uh} + \dot{\uh}^T H_{v_iuu}\frac{\partial \dot{\uh}}{\partial v_j}.
\end{align}
\if{
and our expression for $\partial_{v_j} \ddot{H}_{v_i}$ becomes
\begin{equation}
	\frac{\partial \ddot{H}_{v_i}}{\partial v_j} = p\cdot[f_j, [f,f_i]] +  p\cdot D_u[f,f_i] \frac{\partial \dot{u}}{\partial v_j} + \frac{\partial}{\partial v_j}\frac{\dd}{\dd t}\left(H_{v_iu}\right)\dot{u}.
\end{equation}
}\fi
The proof of identity \eqref{appendix_identity} is organized in the following 3 claims.
\begin{claim}
	\label{claim1}
	The entries of the matrix $R = f_v^TH_{xx}f_v - \left(H_{vx}B + (H_{vx}B)^T\right) - \dot{S}$, given in equation \eqref{matrixM}, satisfy 
	\begin{equation*}
		R_{ij} = -\left\{ \ph\cdot[f_j, [f, f_i]] + \ph\cdot \left(
		\frac{\partial^2 f_i}{\partial x\partial u}f_j - \frac{\partial f_j}{\partial x}\frac{\partial f_i}{\partial u}
		\right)\dot{\uh} \right\}.
	\end{equation*}
\end{claim}
\begin{claim}
	\label{claim2}
	It holds 
	\begin{equation*}
		\frac{\partial \dot{\uh}}{\partial v_j} = - H_{uu}^{-1}E^T_{(:,j)},
	\end{equation*}
	where the matrix $E = f_v^TH^T_{ux} - H_{vx}f_u$ was introduced in \eqref{matrixM}.
\end{claim}
\begin{claim}
	\label{claim3}
	For the matrix $E$ given in \eqref{matrixM}, the following expression holds
	\begin{equation*}
		- E_{(i,:)} = {\ph\cdot D_u [f,f_i] + \dot{\uh}^TH_{v_iuu}}.
	\end{equation*}
\end{claim}
\noindent
{\em Proof of Claim \ref{claim1}.}
	In our case, where we assume uniqueness of multipliers, the matrix $S$ given in \eqref{matrixG} takes the form $S = H_{vx}f_v$, since $H_{vx}f_v$ is symmetric due to Goh conditions. For $i,j=1,\dots,m,$ we obtain
	\begin{equation}\label{dotS}
	\begin{split}
	\dot{S}_{ij} = \frac{\dd}{\dd t}\left(\ph\cdot\frac{\partial f_i}{\partial x}f_j\right) 
	& = \ph\cdot\left[f, \frac{\partial f_i}{\partial x}f_j\right] + D_u \ph\cdot \frac{\partial f_i}{\partial x}f_j \dot{\uh}\\
	& = \ph\cdot\left[f, \frac{\partial f_i}{\partial x}f_j\right] + \ph\cdot\left( \frac{\partial f_i}{\partial x}\frac{\partial f_j}{\partial u} + \frac{\partial^2 f_i}{\partial x\partial u}f_j\right) \dot{\uh}.
	\end{split}
	\end{equation}
	We will make use of the following expression that comes directly from the definition of Lie brackets:
	$ \disp 
	\ph\cdot\frac{\partial f}{\partial x}f_i =\ph\cdot \frac{\partial f_i}{\partial x}f +\ph\cdot [f_i, f],\text{ for $i = 1,\dots, m$}.
	$ Clearly, the additional term $\ph\cdot[f_i, f]$ vanishes, however, as we have discussed, we cannot neglect it once we take partial derivatives
	Summing and subtracting the term $\disp \ph\cdot \frac{\partial^2 f}{\partial x^2}f_if_j$ from the expression for $\disp \ph\cdot  \left[f, \frac{\partial f_i}{\partial x}f_j\right]$ we obtain
	\begin{align*}
	\ph\cdot  \left[f, \frac{\partial f_i}{\partial x}f_j\right] &= \ph\cdot\left(\frac{\partial}{\partial x}\left(\frac{\partial f_i}{\partial x}f_j \right)f - \frac{\partial f}{\partial x}\frac{\partial f_i}{\partial x}f_j \pm \frac{\partial^2f}{\partial x^2}f_if_j\right)\\
	\if{
	&=  \ph\cdot\left(\frac{\partial}{\partial x}\left(\frac{\partial f_i}{\partial x}f_j \right)f - \frac{\partial}{\partial x}\left(\frac{\partial f}{\partial x}f_i \right)f_j + \frac{\partial^2 f}{\partial x^2}f_if_j\right)\\
	}\fi
	&= \ph\cdot\left(\frac{\partial}{\partial x}\left(\frac{\partial f_i}{\partial x}f_j \right)f - \frac{\partial}{\partial x}\left(\frac{\partial f_i}{\partial x}f + [f_i,f] \right)f_j + \frac{\partial^2 f}{\partial x^2}f_if_j\right)\\
	&= \ph\cdot\left( \frac{\partial f_i}{\partial x}[f,f_j]   + \frac{\partial}{\partial x}[f,f_i] f_j + \frac{\partial^2 f}{\partial x^2}f_if_j\right)
	\end{align*}
	Hence, from \eqref{dotS}, we have
	\begin{align*}
		\dot{S}_{ij} = \left(f_v^TH_{xx}f_v\right)_{ij} + \ph\cdot \frac{\partial f_i}{\partial x}[f,f_j]   + \ph\cdot\frac{\partial}{\partial x}[f,f_i] f_j + \ph\cdot\left( \frac{\partial f_i}{\partial x}\frac{\partial f_j}{\partial u} + \frac{\partial^2 f_i}{\partial x\partial u}f_j\right) \dot{\uh}.
	\end{align*}
	Moving on to the terms $\left(H_{vx}B\right)_{ij}$ and $\left(H_{vx}B\right)^T_{ij} = \left(H_{vx}B\right)_{ji}$, and recalling the definition of $B = f_xf_v - \frac{\dd}{\dd t}f_v$, given in \eqref{matrixB}, we obtain that the column of index $j$ for this matrix assumes the form
	$
		B_{(:,j)} = -\left([f,f_j] + \frac{\partial f}{\partial u}\dot{\uh}\right),
	$ 
	so that
	\[
		H_{vx(i,:)}B_{(:,j)} = - \ph\cdot\frac{\partial f_i}{\partial x}\left([f,f_j] + \frac{\partial f}{\partial u}\dot{\uh}\right).
	\]
	Summing all terms to get the matrix $R$, we obtain the desired identity.
\findem 

\noindent
{\em Proof of Claim \ref{claim2}.}
	To obtain an expression for $\disp \frac{\partial \dot{\uh}}{\partial v_j}$, we start by solving the equation $\dot{H}_u = 0$ for $\dot{\uh}$. We obtain 
	\if{
	\begin{align*}
		0 = \dot{H}_u &= H_{ux}\dot{x} + H_{up}\dot{p} + H_{uu}\dot{u}\\
		&= H_{ux}f -\frac{\partial f^T}{\partial u}H_x^T + H_{uu}\dot{u},
	\end{align*}
	obtaining }\fi
	\begin{equation*}
		\dot{\uh} = -H_{uu}^{-1}\left(H_{ux}f -\frac{\partial f^T}{\partial u}H_x^T\right).
	\end{equation*}
	Taking the partial derivative w.r.t. $v_j$ in the latter equation yields
	\begin{align*}
		\frac{\partial \dot{\uh}}{\partial v_j} =& -H_{uu}^{-1}\underbrace{\left(H_{ux}f_j -\frac{\partial f^T}{\partial u}H_{xv_j}^T \right) }_{ = E^T_{(:,j)}} - H_{uu}^{-1}\underbrace{\left(H_{v_jux}f -\frac{\partial f_j^T}{\partial u}H_x^T\right)}_{ = \frac{\partial}{\partial x}H_{v_ju}\dot{\xh} + \frac{\partial}{\partial p}H_{v_ju}\dot{\ph} = -H_{uuv_j}\dot{\uh}}\\
		&  -\frac{\partial H_{uu}^{-1}}{\partial v_j}\underbrace{\left(H_{ux}f_j - \ph\cdot \frac{\partial f_j}{\partial x}f_u\right)}_{ = - H_{uu}\dot{\uh}}\\
		=& -H_{uu}^{-1}E^T_{(:,j)}+ \underbrace{\left(H_{uu}^{-1}\frac{\partial H_{uu}}{\partial v_j} + \frac{\partial H_{uu}^{-1}}{\partial v_j}H_{uu} \right)}_{= \frac{\partial }{\partial v_j}H_{uu}^{-1}H_{uu} = 0 }\dot{\uh}
		= -H_{uu}^{-1}E^T_{(:,j)}.
	\end{align*}
\findem 

\noindent
{\em Proof of Claim \ref{claim3}.}
	Let us expand $D_u\left(\ph\cdot [f, f_i]\right)$:
	\begin{align*}
		D_u\left(\ph\cdot [f, f_i]\right) &= \frac{\partial}{\partial u}\left(\ph\frac{\partial f_i}{\partial x}f - \ph\frac{\partial f}{\partial x}f_i\right)\\
		&= \underbrace{\ph\cdot\frac{\partial f_i}{\partial x}\frac{\partial f}{\partial u} - f_i^TH_{xu}}_{= -E_{(i,:)}} + \underbrace{f^TH_{v_ixu} - H_x\frac{\partial f_i}{\partial u}}_{= \left(\frac{\partial H_{v_iu}}{\partial x}\dot{\xh} + \frac{\partial H_{v_iu}}{\partial p}\dot{\ph}\right)^T}= - E_{(i,:)} - \dot{\uh}^TH_{v_iuu}.
	\end{align*}
\findem 

Finally, we add the contributions of all these claims to prove Lemma \ref{goh_computations}.

\noindent{\em Proof of Lemma \ref{goh_computations}.}
It suffices to check the expression for $\partial v_j\ddot{H}_{v_i}$ in \eqref{DvddotHv}:
\begin{align*}
	\frac{\partial \ddot{H}_{v_i}}{\partial v_j} &= \underbrace{\ph\cdot[f_j, [f, f_i]] + \ph\cdot \left(
		\frac{\partial^2 f_i}{\partial x\partial u}f_j - \frac{\partial f_j}{\partial x}\frac{\partial f_i}{\partial u}
		\right)\dot{\uh}}_{-R_{ij}} + \underbrace{{\ph\cdot D_u [f,f_i] + \dot{\uh}^TH_{v_iuu}}}_{-E_{(i,:)}}\underbrace{\frac{\partial \dot{\uh}}{\partial v_j}}_{-H_{uu}^{-1}E^T_{(:,j)}}\\
	&= -\left(R - EH^{-1}_{uu}E^T\right)_{ij}.
\end{align*}
This concludes the proof. 
\findem

\subsection{}{\em Proof of Lemma \ref{LS-LQS_equivalence}}

	We must check that given a solution $(\xb, \ub, \vb, \pb, \beta)$ of (LS), the corresponding transformed variables $(\xib, \ub, \yb, \hb, \chib, \chib_h, \beta^{LQ})$ solve (LQS).
	
	Starting with the state $\xib$, we recall the dynamics of the linearized variable $\xb$ given in \eqref{statedynamics.linearized} so that one has $\dot{\xib} = \dot{\xb} - \dot{f}_v\yb - f_v\dot{\yb} = f_x\xib + f_u\ub + B\yb,$ retrieving the dynamics in \eqref{LQ.dynamics}. The initial conditions are trivially satisfied since $\yb(0) = 0$. The dynamics for $\hb$ are satisfied by the definition. 
	For the costate dynamics we recall the dynamics of the linearized costates from \eqref{costatedynamics.linearized} and the definition of the matrix $M$ in \eqref{matrixM}. We get
	\begin{align*}
	-\dot{\chib} &= - \dot{\pb} - \dot{\yb}^TH_{vx} - \yb^T\dot{H}_{vx}\\
	&= \underbrace{(\pb + \yb^TH_{vx})}_{=\chib}f_x + \underbrace{(\xb - f_v\yb)^T}_{=\xib^T}H_{xx} + \yb\underbrace{(f_v^TH_{xx} - \dot{H}_{vx} - H_{vx}f_x)}_{=M}\\
	&= \chib f_x + \xib^TH_{xx} + \yb^TM.
	\end{align*}
	Hence the dynamics of $\chib$ matches \eqref{LQ.costatedynamics}. From equation \eqref{mappingLS-LQS} we obtain $\chib(0) = \pb(0)$ and deduce \eqref{LQ.costateinicial}. For the final conditions one substitutes the expressions for $\xb(T)$ and $\pb(T)$ into \eqref{linHvT} and conclude since $S = H_{vx}f_v = f^T_vH_{vx}^T,$	which is a consequence of the Goh conditions \eqref{Goh.condition}.This way we recover the transversality condition for $\chib(T)$. 
	
	Finally we must check the stationarity \eqref{LQ.stationary_u} and \eqref{LQ.stationary_y} of the Hamiltonian for (LQS). Starting from \eqref{linHu} and \eqref{mappingLS-LQS}, we obtain
	\begin{align*}
	0 &= (\bar\chi - \yb^TH_{vx})f_u + (\bar\xi + f_v\yb)^TH_{ux}^T +   \ub^TH_{uu}\\
	&= \bar\chi f_u + \bar\xi^TH_{ux}^T + \ub^TH_{uu} + \yb^T(\underbrace{f_v^TH_{ux}^T - H_{vx}f_u}_{=E}),
	\end{align*}
	which corresponds to the stationarity with respect to $\ub$. On the other hand, the same substitutions applied to \eqref{linHv} yield $0 = \chib f_v + \xib^TH_{vx}^T.$
	Differentiating with respect to time and using the definitions of $B$ in \eqref{matrixB} and $E$ in \eqref{matrixM}, we recover the stationarity \eqref{LQ.stationary_y} with respect to $\yb$. This shows that the tuple $(\xib, \ub, \yb, \hb, \chib, \chib_h, \beta^{LQ})$ is a solution of (LQS) and concludes the proof.
\findem

\subsection{} {\em Proof of Lemma \ref{CP-TP-relation}.}
	Since $\wh$ is a Pontryagin minimum of (CP), from Definition \ref{pontryagin_minimum}, there exists $\varepsilon > 0$ such that
	\begin{equation}
	\label{ineqw}
	\norm{x - \xh}_{\infty} < \varepsilon, \ \norm{(u,v) - (\uh,\vh)}_{1} < \varepsilon, \ \norm{(u,v) - (\uh,\vh)}_{\infty} < 1.
	\end{equation}
	Let $\Wh$ be the transformation of $\wh$ through \eqref{transform.CP-TP}. We now prove that $\Wh$ is weakly optimal for (TP). Hence we search appropriate $\bar\delta, \bar\varepsilon$ for which all feasible trajectories $W = \big((x^k), (u^k), (v^k), (T_k)\big)$ of (TP) that satisfy
	\begin{equation}
	\label{weak_optimality_W}
	\left| T_k - \Th_k\right|<\bar\delta, \quad \norm{(u^k,v^k) - (\uh^k,\vh^k)}_{\infty} < \bar\varepsilon, \ \text{ for all } k = 1, \cdots, N 
	\end{equation}
	will be mapped into a neighborhood of $\wh$ where it is optimal. Such mapping $W \mapsto w$ is done as follows
	\begin{gather}
	\label{transform.TP-CP.xu}
	x(t) := x^k\left(\frac{t - T_{k-1}}{T_k - T_{k-1}}\right), \quad u(t) := u^k\left(\frac{t - T_{k-1}}{T_k - T_{k-1}}\right),\quad \text{ for $t \in I_k$},\\
	\label{transform.TP-CP.v}
	v_i(t) :=
	\left\{
	\begin{array}{cc}
	0,& \text{if $t \in I_k$ and $i\in A_k,$}\\ 
	v_i^k\left(\frac{t - T_{k-1}}{T_k - T_{k-1}}\right),& \text{if $t \in I_k$ and $i\in S_k,$}\\
	1,& \text{if $t \in I_k$ and $i\in B_k.$}\\
	\end{array}
	\right.
	\end{gather}
	
	The dynamics \eqref{state.dynamics} are clearly satisfied by $(x,u,v)$ obtained from \eqref{transform.TP-CP.xu}-\eqref{transform.TP-CP.v}. The end-point constraints in \eqref{initial-final.constraints} are also easy to verify since $x(0) = x^1(0)$ and $x(T) = x^N(1)$ along with the feasibility of $W$.
	
	The last step to check feasibility of $w$ are the control constraints. For the nonlinear controls, note that since $\norm{u^k - \uh^k}_{\infty} < \bar\varepsilon$, we have that $\norm{u - \uh}_{\infty} < \bar\varepsilon$. Recalling $\rho'$ given in \eqref{control_discontinuities}-\eqref{control_set.rho-Ball}, if we choose $\bar\varepsilon < \rho'$, then $u\left([0,T]\right) \subset U$.
	To discuss the feasibility of the linear controls, from equation \eqref{control_discontinuities}, we can choose $\bar\varepsilon$ so that, whenever $t \in I_k$ and $i \in S_k,$
	\begin{equation}
	0 < \rho' - \bar\varepsilon \le v_i(t) \le 1 -\rho' + \bar\varepsilon < 1.
	\end{equation}
	On the other hand, for $i \in A_k \cup B_k$, we know that $v_i(t) \in \{0,1\}$ in view of \eqref{transform.TP-CP.v}, so that the control constraints are still satisfied. This concludes the proof of the feasibility of $(x,u,v)$.
	
	In the sequel, we find $\bar\delta$ and $\bar\eps$ so that, if $W$ satisfies \eqref{weak_optimality_W}, then the transformed $w$ verifies \eqref{ineqw} for the given $\eps.$ The analysis is analogous for both controls $u$ and $v$, hence we will conduct the calculations only for $u$. We have
	{\small
		\begin{equation}
		\begin{array}{cc}
		\disp \int_{I_k\cap \Ih_k} |u_i(t) - \uh_i(t)|\dd t & \disp \le
		\int_{I_k\cap \Ih_k} \left| u_i^k\left(\frac{t - T_{k-1}}{T_k - T_{k-1}}\right) - \uh_i^k\left(\frac{t - T_{k-1}}{T_k - T_{k-1}}\right)\right|\dd t \\
		
		& \disp + \int_{I_k\cap \Ih_k} \left| \uh_i^k\left(\frac{t - T_{k-1}}{T_k - T_{k-1}}\right) - \uh_i^k\left(\frac{t - \Th_{k-1}}{\Th_k - \Th_{k-1}}\right)\right|\dd t.
		\end{array}
		\end{equation}}
	The first integral in the r.h.s. of latter display is bounded by $\bar\varepsilon|I_k\cap \Ih_k|$ in view of \eqref{weak_optimality_W}. For the second term, recall that $\uh$ is continuous on $[0,T]$ and so are the components of $\uh^k$ over $\Ih_k$, so that they are uniformly continuous over these intervals. Therefore, for each $k = 1, \cdots, N,$ we can find some $\bar\delta_k>0$ such that, if $|T_k - \Th_k| < \bar\delta_k,$ then
	\begin{equation*}
	\left| \uh_i^k\left(\frac{t - T_{k-1}}{T_k - T_{k-1}}\right) - \uh_i^k\left(\frac{t - \Th_{k-1}}{\Th_k - \Th_{k-1}}\right) \right| < \bar\varepsilon
	\end{equation*}
	for every component of $\uh^k$. Hence we only need to choose $\disp \bar\delta:=\min_{k = 1, \cdots, N} \bar\delta_k$. We proved that
	\be
	\label{estimate_u_1}
	\disp \int_{I_k\cap \Ih_k} |u_i(t) - \uh_i(t)|\dd t  \leq 2\bar\varepsilon|I_k\cap \Ih_k|.
	\ee
	Next, we need to estimate the integral outside the intersection $I_k\cap \Ih_k.$ We assume w.l.o.g. that $T_k < \Th_k$ hence, in view of \eqref{weak_optimality_W},
	\begin{equation}
	\label{estimate_u_2}
	\int_{T_k}^{\Th_k}|u_i(t) - \uh_i(t)|\dd t \le \bar\delta\bar\varepsilon.
	\end{equation}
	Adding up all the terms, we get from \eqref{estimate_u_1}-\eqref{estimate_u_2}, that
	$$\norm{u_i - \uh_i}_1 < \bar\varepsilon(2T + (N-1)\bar\delta).
	$$ 
	An analogous estimate can be obtained for $\|v-\vh\|_1.$ Finally, taking into account all the control components  $m$ of the linear controls and $l$ from the nonlinear controls, we get that, if
	\begin{equation*}
	\bar\varepsilon(2T + (N-1)\bar\delta) < \frac{\varepsilon}{m + l},
	\end{equation*}
	then $\norm{u - \uh}_1 < \varepsilon$, as desired.
	
\findem


\bibliographystyle{plain}
\bibliography{mixed_controls_bib}


\end{document}